\documentclass{siamart190516}
\usepackage{bbm}
\usepackage[font={small,it}]{caption}

\usepackage{array}
\usepackage{lipsum}
\usepackage{amsfonts}
\usepackage{graphicx}
\usepackage{epstopdf}
\usepackage{algorithmic}
\usepackage{enumerate}
\usepackage{subcaption}
\usepackage{multirow}

\ifpdf
  \DeclareGraphicsExtensions{.eps,.pdf,.png,.jpg}
\else
  \DeclareGraphicsExtensions{.eps}
\fi

\newsiamremark{remark}{Remark}
\newsiamremark{hypothesis}{Hypothesis}
\crefname{hypothesis}{Hypothesis}{Hypotheses}
\newsiamthm{claim}{Claim}
\newsiamthm{defin}{Definition}

\headers{An Optimal Mass Transport Method for Random Genetic Drift}{Jos{\'e} A. Carrillo, Lin Chen, and Qi Wang}

\title{An Optimal Mass Transport Method for Random Genetic Drift}

\author{Jos{\'e} A. Carrillo\thanks{Mathematical Institute, University of Oxford, Oxford OX2 6GG, United Kingdom;
  (\email{carrillo@maths.ox.ac.uk}).}
\and Lin Chen\thanks{Department of Mathematics, Southwestern University of Finance and Economics, 555 Liutai Ave, Wenjiang, Chengdu, Sichuan 611130, China
  (\email{lchen@smail.swufe.edu.cn }).}
\and Qi Wang\thanks{Department of Mathematics, Southwestern University of Finance and Economics, 555 Liutai Ave, Wenjiang, Chengdu, Sichuan 611130, China
  (\email{qwang@swufe.edu.cn }).}}

\usepackage{amsopn}

\ifpdf
\hypersetup{
  pdftitle={An Optimal Mass Transport Method for Random Genetic Drift},
  pdfauthor={Jos{\'e} A. Carrillo, Lin Chen, and Qi Wang}
}
\fi

\begin{document}

\maketitle

\begin{abstract}
We propose and analyze an optimal mass transport method for a random genetic drift problem driven by a Moran process under weak-selection.  The continuum limit, formulated as a reaction-advection-diffusion equation known as the Kimura equation, inherits degenerate diffusion from the discrete stochastic process that conveys to the blow-up into Dirac-delta singularities hence brings great challenges to both the analytical and numerical studies. The proposed numerical method can quantitatively capture to the fullest possible extent the development of Dirac-delta singularities for genetic segregation on one hand, and preserves several sets of biologically relevant and computationally favored properties of the random genetic drift on the other.  Moreover, the numerical scheme exponentially converges to the unique numerical stationary state in time at a rate independent of the mesh size up to a mesh error.  Numerical evidence is given to illustrate and support these properties, and to demonstrate the spatio-temporal dynamics of random generic drift.
\end{abstract}



\section{Introduction}

In population genetics, genetic drift describes random fluctuations in the numbers of gene variants (alleles) over time.  Allele frequency, expressed as a percentage, measures the relative fraction of an allele at a particular locus in the population, and its change quantifies the intensity of the random genetic drift \cite{Gillespie2004}.  Typically, when genetic drift begins, it will continue until either i) the involved allele completely disappears from the population or ii) the allele establishes permanently at 100\% frequency (called fixed).  In either case, genetic drift causes gene variants to disappear because infrequently occurring alleles face a greater chance of being lost in a small population, or causes a new population genetically distinct from its original population such that initially rare alleles become much more frequent and even fixed \cite{Gillespie2004,Rice1987,StarSpencer2013,Waxman2009}.  Both events indicate that genetic drift can decrease the population's genetic diversity, and it plays a role in the evolution of new species.

Mathematical modeling of genetic drift dates back to the pioneering works of Ronald Fisher \cite{Fisher1922} and Sewall Wright \cite{Wright1929,Wright1937,Wright1945}.  The Wright--Fisher model employs a discrete stochastic process to model dynamics of finite populations at the individual level under the restrictions that the generations do not overlap and that each copy of the gene of the new generation is selected independently and randomly from the whole gene pool of the previous generation.  It is later modified and extended by Patrick Moran \cite{Moran,Moran1962} (allowing generations overlap) and Komoo Kimura \cite{Kimura1955a,Kimura1955b,Kimura1962,Kimura1964} (allowing the genetic mutation to spread across the population).  In particular, they show that, in the limit of a large population and weak selection, these processes can be approximated by the same diffusion approximation, namely the Kimura equation \cite{Kimura1962} which describes the probability of fixation of a mutant with frequency-independent fitness.  The continuum framework makes a systematic qualitative and quantitative analysis of the new model possible thanks to the tools from modern analytical and numerical analysis.

In this paper, we consider the Moran process as a paradigm and introduce its large population limits with different drift-diffusion scalings assumption \cite{Chalub2009}.  Consider the dynamics of a population with $N$ individuals that are distinguished by two neutral alleles labeled $A$ and $B$ such that they do not affect the survival and reproduction ability of the individual.  In light of the balance between selection and drift, Traulsen \emph{et al.} \cite{Traulsen2005} summarized the process into three simple steps: (a) \emph{selection}--an individual is randomly selected for reproduction with a probability proportional to its fitness;  (b) \emph{reproduction}--the selected individual produces one identical offspring; (c) \emph{replacement}--the offspring replaces a randomly selected individual in the population. The process is then repeated after each time step $\Delta t$.

\begin{figure}[H]
  \centering
  \includegraphics[width=0.32\textwidth]{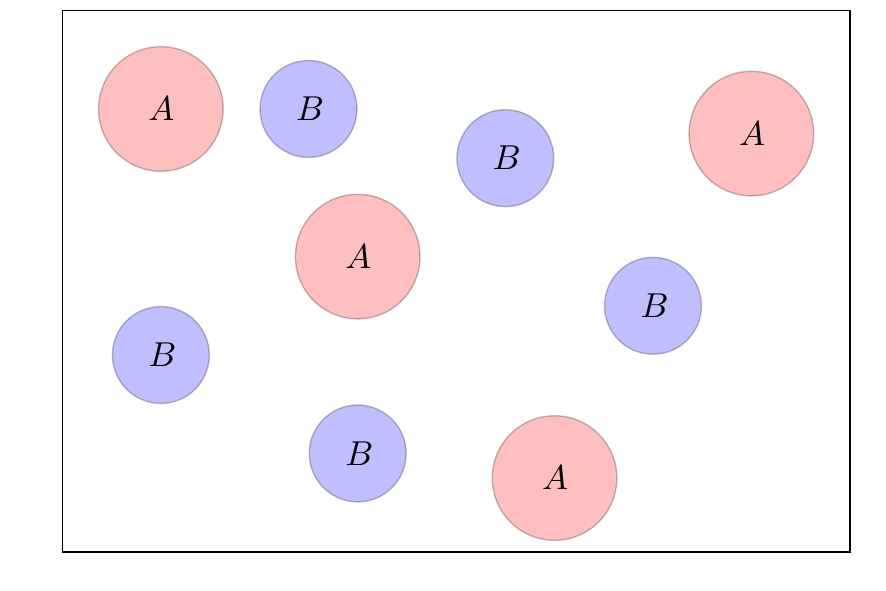}
    \includegraphics[width=0.32\textwidth]{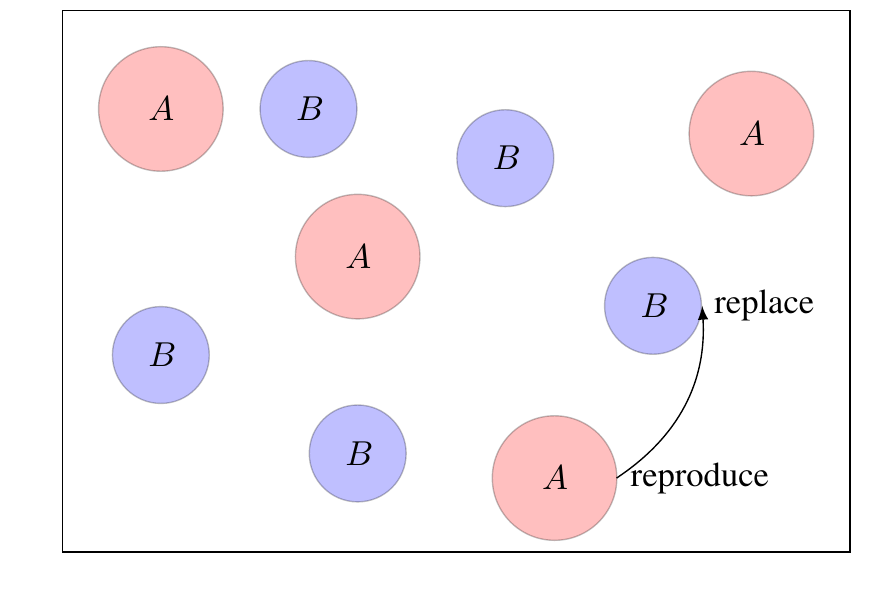}
      \includegraphics[width=0.32\textwidth]{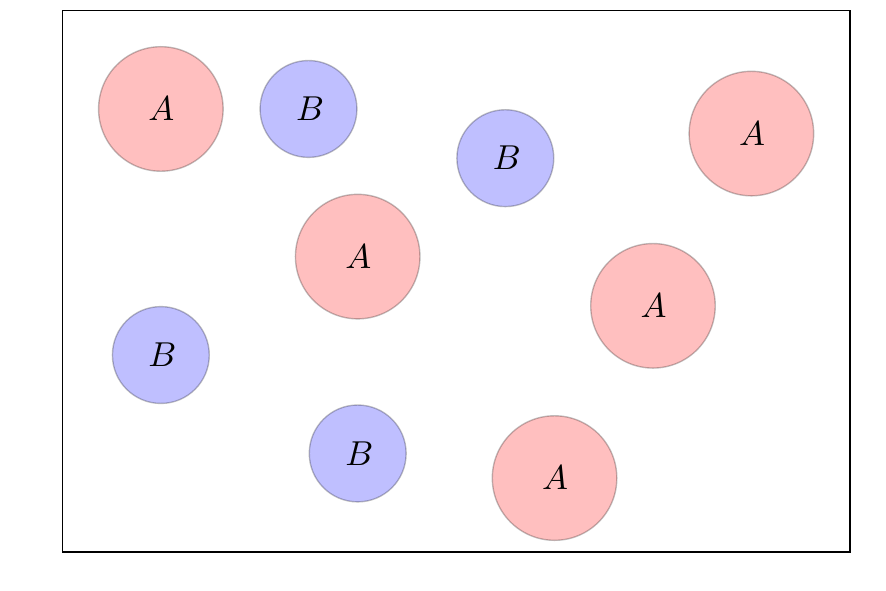}
  \caption{Mutation between two alleles $A$ and $B$ through a Moran process.}
\end{figure}

The fitness or reproduction rate of an individual depends on environmental conditions as well as the type and frequency of its competitors. For each type, we associate a fitness function depending on the type frequency:
$$\Psi^{(A)}(x;N,\Delta t),\; \Psi^{(B)}(x;N,\Delta t): [0,1]\rightarrow \mathbb R^+,$$
through the weak selection principle for $i=A,B$
\begin{equation*}
  \Psi^{(i)}(x;N,\Delta t)\approx 1+(\Delta t)^{\nu} \psi^{(i)}(x)+o((\Delta t)^{\nu}), \quad N \approx \infty,\;\Delta t \approx  0.\; 
\end{equation*}
In particular, when $\Psi$ is constant, i.e., the fitness of an individual is genetically determined and not affected by interactions, one recovers the classical frequency-independent Moran process.

Under the assumption $N^{-1}\propto (\Delta t)^{\mu}$ and a suitable time rescaling, as $N\rightarrow \infty$, one collects from \cite{Chalub2009,Kimura1964,Traulsen2005} the following thermodynamical limit for the density $f(x,t)$ of allele $A$ when $\nu=\mu =\frac{1}{2}$
\begin{equation}\label{driftdiff}\tag{1.1}
f_t=\frac{\kappa}{2}\big(x(1-x)f\big)_{xx}+\big(x(1-x)V'(x)f\big)_x,\qquad \qquad-\text{\textbf{Replicator-Diffusion equation}}
\end{equation}
where $\frac{1}{N(\Delta t)^\mu}\rightarrow\kappa>0$ is a constant, and $V(x):[0,1]\rightarrow \mathbb{R}$ is called the fitness potential such that $V'(x):=\psi^{(B)}(x)-\psi^{(A)}(x)$ measures the fitness difference between the focal and opponent.  (\ref{driftdiff}) nests the following purely diffusive or advective equation with $V\equiv0$ or $\kappa=0$
\begin{subequations}
\begin{align}
f_t&=\big(x(1-x)f\big)_{xx}, \qquad \qquad \qquad \qquad -\text{\textbf{Diffusion equation}}\label{diff}\tag{1.1a}\\ f_t&=\big(x(1-x)V'(x)f\big)_x, \qquad\qquad \qquad -\text{\textbf{Replicator equation}}\label{drift}\tag{1.1b}
\end{align}
\end{subequations}
as two special cases such that genetic drift is the only evolutionary force in (\ref{diff}), and the evolutionary force includes genetic mutation, migration and selection in (\ref{drift}).  They correspond to the limits of discrete process with $\nu >\mu =\frac{1}{2}$ and $\mu \in (\frac{1}{2},1]$, $\nu=1-\mu$, respectively. 

The biological significance urges us to impose the non-flux boundary condition to each of \eqref{driftdiff} 
\begin{equation}\label{NBC}
\frac{\kappa}{2}\big(x(1-x)f\big)_{x}+\big(x(1-x)V'(x)f\big)=0, \quad x=0,1, \forall t>0
\end{equation}
such that the following conservation holds
\begin{equation}\label{conser0}
    \frac{d}{dt} \int_{0}^{1} f(x,t) dx =0, \forall t>0
\end{equation}
and leads to well-defined evolution of the probability measure.  Moreover, a second conservation law applies to the replicator-diffusion equation \eqref{driftdiff} and the diffusion equation \eqref{diff}, and it reads
\begin{equation}\label{conser}
  \frac{d}{dt} \int_{0}^{1} \theta(x) f(x,t) dx =0, \forall t>0,
\end{equation}
where $\theta(x)$ is the fixation probability function that satisfies
$$\frac{\kappa}{2} \theta''(x)-V'(x) \theta'(x)=0, \quad \theta(0)=0,\quad \theta(1)=1,$$
and it can be explicitly given by
\begin{equation}\label{thetaV}
  \theta(x)=\frac{\int_{0}^{x} e^{\frac{2}{\kappa} V(y)} dy}{\int_{0}^{1} e^{\frac{2}{\kappa} V(y)} dy}.
\end{equation}
One notes that \eqref{conser} recovers the conservation of total population for \eqref{driftdiff} with $\theta(x)\equiv 1$, and conservation of mass center for \eqref{diff} with $\theta(x)=x$.

In this work, we will take advantage of the fact that the following free energy admits \eqref{driftdiff} as a gradient flow with respect to a variation of optimal transport distances (see e.g. \cite{Vi03,CMV03,CLSS10,Chalub}) 
\begin{equation}\label{entropy}
  E(f)=\frac{\kappa}{2}\int_{0}^{1} f(x,t) \ln \big(x(1-x)f(x,t)\big)dx +\int_{0}^{1}V(x)f(x,t)dx
\end{equation}
and that \eqref{driftdiff} can be rewritten as
\begin{equation}\label{gradflow} 
f_t =- \text{grad}_w {E}(f) \qquad \mbox{with} \quad
\text{grad}_w {E}(f):=-\left(x(1-x)f(x,t) \left(\frac{\delta {E}}{\delta f} \right)_x \right)_x,
\end{equation}
where $\frac{\delta {E}}{\delta f}=\frac{\kappa}{2}\ln(x(1-x)f)+V(x)$ denotes the first variation of the energy $E$ with a fixed mass constraint.

We now recall several theoretically relevant results on the well-posedness of \eqref{driftdiff}-\eqref{NBC} by \cite{Chalub2009a,Chalub2009,Tran2013}.  Let us denote by $\mathcal{BM}^+([0,1])$ the space of functions with positive Radon measure in $[0,1]$.  A function $f(x,t)\in L^{\infty}([0,\infty);\mathcal{BM}([0,1]))$ is called a \emph{weak solution} of (\ref{driftdiff}) and (\ref{NBC}) according to \cite{Chalub2009} if the following identity holds for any test function $\zeta(x,t)\in C_c^{\infty} ([0,\infty)\times [0,1])$
\begin{equation*}
  -\int_{0}^{\infty}\int_{0}^{1} f(x,t)\zeta_t(x,t) dx dt  =\int_{0}^{\infty}\int_{0}^{1} x(1-x)\left( \frac{\kappa}{2}\zeta_{xx}(x,t)-V'(x)\zeta_x(x,t) \right) f(x,t) dx dt
  +\int_{0}^{1} f_0(x) \zeta(x,0) dx.
\end{equation*}
Then \eqref{driftdiff} with conservation laws \eqref{conser0}-\eqref{conser} is well-posed as follows according to \cite{Chalub2009a}:
\begin{theorem}\label{theorem11}
For any given $f_0 \in \mathcal{BM}^+([0,1])$, \eqref{driftdiff} under \eqref{NBC} admits a unique weak solution $f(x,t)$ such that $f\in L^{\infty}\big([0,\infty);\mathcal{BM}^+([0,1])\big)\cap C^{\infty}\big( \mathbb R^+,C^{\infty}((0,1)) \big)$, and it satisfies the conservation laws \eqref{conser0}-\eqref{conser}.  Moreover, the solution can be written as
  $$f(x,t)=r(x,t)+a(t) \delta_0(x) +b(t) \delta_1(x),$$
  where $r\in C^{\infty}\big(\mathbb R^+;C^{\infty}([0,1]) \big)$ is the classical solution to \eqref{driftdiff} without boundary conditions, functions $a(t)$ and $b(t)\in C([0,\infty)) \cap C^{\infty} (\mathbb R^+)$ are monotonically increasing, and $\delta_y$ is the singular measure supported at $y$.
Furthermore, as $t\rightarrow \infty$, $r(x,t)\rightarrow0$ uniformly and
$$ f(\cdot,t)\rightarrow f_\infty(x):=\left(1-\int_{0}^{1}f_0(x) \theta(x)dx \right) \delta_0(x) +\left(\int_{0}^{1}f_0(x) \theta(x)dx\right)\delta_1(x),$$
exponentially fast with respect to a transport metric, where $\theta$ is given by \eqref{thetaV}.
\end{theorem}

According to Theorem \ref{theorem11}, we always expect (\ref{driftdiff}) to collapse into the Dirac-delta singularity regardless of the initial distribution.  The long-time dynamics along with this spiky spatial profile well demonstrate and capture the formation of gene segregation, i.e., an asymptotic gene fixation with allele $A$ (resp. allele $B$) occupying the whole population when $\int_0^Lf_0(x)\theta(x)dx=0$ (resp. 1).

More recently, the gradient flow structure \eqref{gradflow} has been discussed in \cite{Chalub}.  The analysis of the model is built upon the classical steepest descent variational schemes for nonlinear Fokker--Planck equations which are introduced in \cite{gf-JKO} and generalized in \cite{gf-AGS,BCC2008}.  The transport distance between probability measures has to be adapted to the degenerate diffusion coefficient of \eqref{driftdiff}.  They introduced a suitable distance called the Shahshahani distance showing the convergence of the variational scheme for the specific case of replicator dynamics, i.e. \eqref{driftdiff} without the degenerate diffusion term. It is an open problem to show the convergence of this variational scheme including the degenerate diffusion term.


The spatial-temporal dynamics of the Kimura equation are well understood in the purely diffusive case.  However, the singularity induced by the convergence towards Dirac-delta concentrations, which well model the biologically realistic gene segregation or fixation through the drift, imposes substantial challenges to their numerical approximation, in particular for long-time computations. Therefore, one of the crucial features of any numerical scheme to solve \eqref{gradflow}, while keeping its biological relevance, is to accurately capture the concentration phenomena at the discrete level.
Lagrangian numerical schemes for these gradient flow interpretations of one-dimensional Fokker--Planck, thin-film and quantum drift equations have been devised by various authors, see e.g. \cite{gf-GT1,gf-GT2,BCC2008,gf-WW,gf-CN,CRW2016,gf-MOfp,gf-MOtf,gf-MOdlss,gf-Osberger}. We here build upon the strategy of solving an equivalent equation satisfied by the diffeomorphisms mapping the initial data to the solution at later times as performed in \cite{BCC2008,CM2009,CRW2016,CKM2019}, see next section for details.  We refer to \cite{CMW20} for a very recent review on the state of the art of this kind of numerical schemes. The main advantage of the present approach is being able to deal with Dirac-delta concentrations easily and rigorously now that we work in mass variables. This merit is also utilized by \cite{BCC2008,CRW2016,CKM2019} to study blow-ups in Keller--Segel models for chemotactic cellular aggregation.

Other numerical schemes have recently been proposed in the literature to solve the Kimura equation \eqref{driftdiff}. In \cite{DLWY2019} the authors introduced a numerical method based on Lagrangian maps that preserves the free energy decay of the system.  They also analyze a convex-concave splitting approach which leads to an implicit method, and prove the unique solvability of this method.  However, the equation is solved in the original variables therein, so they have to devise heuristic criteria to capture the concentration of mass towards the endpoints.  Moreover, they do not take full advantage of the gradient flow structure \eqref{gradflow} and the variational schemes at the discrete level.  As an alternative, our approach is to construct numerical schemes directly on the optimal maps and can describe the Dirac-delta to the fullest extent.  See details in Section 2.  \cite{XCLZY20192019} performs an interesting ``horse race" comparison of a serial of finite volume and finite element schemes for (\ref{diff}).  Their critical comparison of the long-time asymptotic performance urges carefulness in choosing a numerical method for this type of problem, especially when the main properties of the model are not kept by the scheme.  We also want to mention that the genetic drift problem with multi-alleles, cast as a multi-dimensional PDE, is investigated in \cite{XCLY2019} through finite-difference methods, where the authors propose a numerical scheme with absolute stability and conserve several biologically/physically relevant quantities such as positivity, total probability, and conservation.   

The rest of this paper is organized as follows.  Section 2 introduces the evolution problem for the map in mass variables as in \cite{BCC2008,CRW2016,CMW20} together with the needed prerequisites in optimal transport theory. Since the free-energy functional shows again in the new variables a non-convex structure, we present our numerical method and then apply it for the genetic-drift problem by introducing the convex-splitting technique as the preprocessing step of the Euler implicit scheme.  Section 3 is devoted to the analysis of the properties of the numerical scheme. We show that the discrete problem converges exponentially fast to a unique stationary state with a monotone drift $V$ for any fixed discretization parameter modulo error terms in spatial discretization, and we characterize all possible stationary states and limit points otherwise.   See \cref{theorem_decay1} and \cref{theorem_decay2} for our main results.  Finally, Section 4 conducts several sets of numerical experiments to validate the shown properties and to accurately illustrate the long-time dynamics of random genetic drift.

\section{Numerical Methods}
In this section, we recast the random generic drift models into an evolution problem for map of the mass variable, and then propose a numerical scheme to solve the new equation.  Since the associated new free-energy functional is non-convex, we introduce the convex-splitting technique as the preprocessing step of the Euler implicit scheme in its numerical solver.

\subsection{Optimal transport and the Wasserstein distance}
We first introduce a suitable Wasserstein distance in the probability space $\mathcal{P}([0,1])$ such that equation \eqref{driftdiff} can be interpreted as a gradient flow of the free energy \eqref{entropy}.
Due to the presence of the variable coefficient $x(1-x)$ in \eqref{driftdiff}, the quadratic Wasserstein distance will not be based on the usual Euclidean distance but on the induced generalized Shahshahani distance
\begin{equation}\label{d}
  d^2(x,y):= \inf_{\substack{\xi\in C^1([0,1];\Omega) \\ \xi(0)=x, \xi(1)=y}} \int_{0}^{1}|\xi'(t)|^2_{T_{\xi(t)}\Omega} dt = \inf_{\substack{\xi\in C^1([0,1];\Omega) \\ \xi(0)=x, \xi(1)=y}} \int_{0}^{1}   \frac{|\xi'(t)|^2}{\xi(t)(1-\xi(t))}  dt,
\end{equation}
for $x,y \in \Omega=(0,1)$.
Chalub \emph{et al.} \cite[Lemma 10]{Chalub} prove that the infimum in \eqref{d} is achieved at a unique constant-speed geodesic, and $d$ can be uniformly extended to $\bar{\Omega}\times \bar{\Omega}$ as
  \begin{equation*}
    d(x,y)=\left| \int_{x}^{y}\frac{d u}{\sqrt{u(1-u)}} \right|=\left| \arcsin (2y-1)-\arcsin(2x-1) \right|, \quad
  x,y \in [0,1];
  \end{equation*}
moreover, $d$ defines a distance in $\bar{\Omega}$ and the metric space $(\bar{\Omega}, d)$ is Polish.
This Shahshahani distance is locally equivalent to the Euclidean one in the interior, but behaves differently close to the boundary. This difference is reflected by the dynamics of replicator-diffusion equation \eqref{driftdiff}, which is locally uniformly parabolic in the interior, but degenerate at the boundaries.

We now present several concepts from transport theory essential for this paper.  Let $\mu$ and $\nu$ be two absolutely continuous measures with respect to the Lebesgue in $\mathcal{P}([0,1])$, and $T$ be a measurable map from $[0,1]\rightarrow [0,1]$.  We say that $T$ transports $\mu$ onto $\nu$ and write $\nu=T\sharp\mu$ if $\nu(B)=\mu \circ T^{-1}(B)$ for any measurable set $B \subset [0,1]$.  We also say $\nu$ is the push-forward or the image measure of $\mu$ by $T$.  To introduce the corresponding Wasserstein distance between $\mu$ and $\nu$, one defines
\begin{equation*}
  d^2_W(x,y):= \inf_{T:\nu=T\sharp\mu} \int_{0}^{1}d^2(x,T(x)) d \mu(x),
\end{equation*}
as soon as the source measure $\mu$ has no atoms.  In fact, by Brenier's theorem, if $\mu$ is absolutely continuous with respect to the Lebesgue measure, then there exists a measurable nondecreasing map $T$ such that $\nu=T\sharp\mu$. The proper definition of the associated Wasserstein distance needs a relaxed variational scheme given by
\begin{equation}\label{dw}
  d^2_W(x,y):= \inf_{\Pi\in\Gamma} \left\{\int_{[0,1]\times [0,1]}\left| \arcsin (2y-1)-\arcsin(2x-1) \right|^2 d \Pi(x,y) \right\},
\end{equation}
where $\Pi$ runs over the set of transference plans $\Gamma$ between $\mu$ and $\nu$. One important simplification in 1D is that the optimal plan can be characterized fully in terms of the inverse of cumulative distribution functions, see \cite{Vi03}.  To be specific, let $F$ and $G$ be the cumulative of the 1D functions $f$ and $g$, and define the pseudo-inverse
$$\Phi(\eta,t):=F^{-1}(\eta,t)=\inf\{x\in[0,1]:F(x,t)>\eta\}.$$
By Brenier's theorem and the definition of the image measure, we have
$$ F(x,t)=\int_{-\infty}^{x} f(y)dy =\int_{-\infty}^{\varphi'(x)} g(y) dy =G\circ \varphi'(x). $$
Then it is straightforward to obtain $\varphi'=G^{-1}\circ F$, and the Wasserstein distance in \eqref{dw} becomes
\begin{equation}\label{wd}
    d_W^2(\mu,\nu)=\int_{0}^{1} d^2(F^{-1}(\omega),G^{-1}(\omega)) d\omega.
\end{equation}
The proof of this fact without relying to Brenier's theorem can be found in \cite[Section 2.2]{Vi03}.

By the definition of pseudo-inverse function, we can derive the following evolution equation satisfied by $\Phi(\eta,t)$
\begin{equation}\label{phi}
\left\{\begin{array}{ll}
\displaystyle \Phi_t =-\frac{\kappa}{2}\Phi(1-\Phi) \frac{\partial}{\partial \eta}\left(\left(\frac{\partial \Phi}{\partial \eta}\right)^{-1} \right)-\frac{\kappa}{2}(1-2\Phi)-\Phi(1-\Phi)V'(\Phi),\; \eta\in(0,1),t>0,\\[5mm]
\Phi(0,t)=0,\; \Phi(1,t)=1,\; t>0.
\end{array}
\right.
\end{equation}
We would like to remark that, the Dirichlet boundary condition of $\Phi(1,t)=1$ applies here since one have from the strong maximum principle that $f(x,t)$ in strictly positive in $[0,1]$ for all time $t$ if $f_0(x)\geq,\not\equiv0$.  This, on the other hand, indicates that $\Phi(\eta,t)$ must be strictly increasing in $\eta$, which is an important property to preserve for numerical schemes.  Actually, according to the non-flux boundary condition of $f(x,t)$, we have
\begin{equation*}
    \frac{\kappa}{2}\big(x(1-x)f\big)_{x}+\big(x(1-x)V'(x)f\big)=0, \quad x=0,1, \forall t>0.
\end{equation*}
If $f_0(x)>0$, we can obtain from the definition of $\Phi(\eta,t)$ that for all $t>0$
\begin{equation*}
    \Phi_t=-\frac{\frac{\partial F}{\partial t}\Big|_{x=\Phi}}{\frac{\partial F}{\partial x}\Big|_{x=\Phi}}=-\frac{\frac{\kappa}{2} \partial_x \big(x(1-x)f \big)_{x=\Phi}+\left(x(1-x)V'(x)f \right)_{x=\Phi}}{f(x,t)_{x=\Phi}}=0, \quad \text{at}~\eta=0,1.
\end{equation*}
This implies the Dirichlet boundary condition holds $\Phi(0,t)=0$ and $\Phi(1,t)=1$ for all time $t$.  However, a free boundary condition should be adopted if $f(x,t)$ remains compactly supported when studying problems with diffusion degenerate inside the domain.  

Note that (\ref{phi}) has the following free energy
\begin{equation}\label{freeenergy}
  \mathcal{E}(\Phi):=-\frac{\kappa}{2}\int_{0}^{1} \ln\left(\frac{\partial \Phi}{\partial \eta} \right) d \eta + \frac{\kappa}{2}\int_{0}^{1} \ln \big(\Phi(1-\Phi) \big) d\eta +\int_{0}^{1}V(\Phi)d\eta.
\end{equation}
Then we will connect this evolution problem for the map $\Phi(\eta,t)$ pushing forward the initial data $f_0$ to the solution $f(x,t)$ at time $t$ with the continuum limit of implicit Euler steps obtained as Euler--Lagrange conditions for suitable variational problems. Finally, let us restate Theorem \ref{theorem11} in terms of the map $\Phi(\eta,t)$.

\begin{theorem}\label{theorem21}
The map $\Phi(\eta,t)$ pushing forward the initial data $f_0$ to the solution $f(x,t)$ of the problem \eqref{driftdiff}-\eqref{NBC} satisfies  that
$$
d_W^2(f(\cdot,t),f_\infty)=\int_{0}^{1} d^2\big(\Phi(\eta,t),\Phi_\infty (\eta)\big) d\eta\to 0 \quad
\mbox{exponentially fast as } t\to\infty,
$$
where
\[
\Phi_\infty (\eta) =\left\{
\begin{array}{ll}
0& \text{for} \; 0\leq\eta\leq\eta_0, \\[3mm]
1&\text{for} \;  \eta_0<\eta\leq1,\\
\end{array} \qquad \text{with~}\eta_0:=1-\int_{0}^{1}f_0(x) \theta(x)dx \mbox{~and~} \theta(x) \mbox{given by \eqref{thetaV}.}
\right.
\]
\end{theorem}

In the sequel, we will design a numerical scheme capable of accurately capturing the long-time behavior described in Theorem \ref{theorem21}.

\subsection{Discretization for Euler--Lagrange of Replicator-Diffusion \cref{driftdiff}}
We now consider the spatio-temporal discretization for the Euler--Lagrange problem \eqref{phi} of the full Replicator-Diffusion equation \eqref{driftdiff}.  Throughout this paper, we assume that the time step $\tau$ and space step $h$ are constant in the discretization.  In terms of the Wasserstein distance \eqref{wd}, the Jordan--Kinderlehrer--Otto (JKO) steepest descent scheme implies that finding the inverse distribution function in \eqref{phi} corresponds to solve the following for a fixed time step $\tau>0$
$$\Phi^{k+1} \in \mathop{\arg\inf}_{\omega:(\omega^{-1})'\in\mathcal A} \left[ \mathcal{E}(\omega) +\frac{1}{2 \tau} d^2_W(\omega,\Phi^k)  \right]$$
over the admissible set $f\in\mathcal A:=\{f\in L^1_+(0,1):f \ln (x(1-x)f) \in L^1(0,1)\}$.  In light of \eqref{dw}, we find that
 \begin{align*}
   \frac{\delta d^2_W(\omega,\Phi^k)}{\delta \omega} &= \frac{d}{d \omega} \left| \arcsin(2 \omega -1)- \arcsin(2 \Phi^k -1)\right|^2 \\
    &= \frac{2}{\sqrt{\omega(1-\omega)}} \left(\arcsin(2 \omega -1)- \arcsin(2 \Phi^k -1) \right)\\
    & \approx \frac{2}{\sqrt{\omega(1-\omega)}} \frac{(\omega-\Phi^k)}{\sqrt{\omega(1-\omega)}}=\frac{2(\omega-\Phi^k)}{\omega(1-\omega)}.
 \end{align*}
 Hence an approximated Euler--Lagrange equation \eqref{phi} associated to this minimization problem is
 \begin{equation}\label{jko}
    \frac{1}{\Phi^{k+1}(1-\Phi^{k+1})} \frac{\Phi^{k+1}-\Phi^k}{\tau} =
   - \frac{\kappa}{2}\frac{\partial}{\partial \eta}\left[ \left(\frac{\partial \Phi^{k+1}(\eta)}{\partial \eta}\right)^{-1} \right]-\frac{\kappa}{2}\frac{1-2\Phi^{k+1}}{\Phi^{k+1}(1-\Phi^{k+1})}-V'(\Phi^{k+1}).
 \end{equation}

If we denote $\Phi_i^k=\Phi(ih,k\tau)$ for $i=0,1,\cdots,N$ and $Nh=1$, $k\in \mathbb{N}$, our full finite difference discretization of \eqref{jko} is the following implicit scheme:
\begin{equation}\label{Euler}
  \frac{1}{\Phi^{k+1}_i(1-\Phi^{k+1}_i)}\frac{\Phi^{k+1}_i-\Phi^k_i}{\tau} =
  -\frac{\kappa}{2}\left(\frac{1}{\Phi^{k+1}_{i+1}-\Phi^{k+1}_i}- \frac{1}{\Phi^{k+1}_{i}-\Phi^{k+1}_{i-1}} \right)-\frac{\kappa}{2}\frac{1-2\Phi^{k+1}_i}{ \Phi^{k+1}_i(1-\Phi^{k+1}_i)}-V'(\Phi_i^{k+1})
\end{equation}
with the Dirichlet boundary condition $\Phi^{k+1}_0=0$ and $\Phi^{k+1}_N=1$.  The solution at each time step is computed by an iterative Newton's procedure with initial diffeomorphism obtained by a preprocessing step, and one does not need a CFL condition for this implicit in time discretization.

\subsubsection{Preprocessing Step via a Convex Splitting Technique}

System \eqref{phi} can be viewed as a gradient flow associated with the energy functional \eqref{freeenergy}.  This energy has a non-convex structure, and it prohibits the direct application of the proposed implicit scheme due to the singularity of the numerical scheme on the boundary and local convergence of Newton's method.  However, the merits brought by the convex-splitting technique enable us to construct a numerical scheme in the preprocessing step and then calculate the initial diffeomorphism for Newton's method in the Euler implicit scheme.

To this end, we write $V=V_c-V_e$ for $V_c$ and $V_e$ being smooth convex functions and then apply the convex splitting method of Eyre \cite{Eyre} to obtain the following semi-discrete scheme
\begin{equation}\label{semi}
  \frac{1}{\Phi^{k}(1-\Phi^{k})} \frac{\Phi^{k+1}-\Phi^k}{\tau} = - \frac{\kappa}{2} \frac{\partial}{\partial \eta}\left[ \left(\frac{\partial \Phi^{k+1}(\eta)}{\partial \eta}\right)^{-1} \right]-\frac{\kappa}{2}\frac{1-2\Phi^{k}}{\Phi^{k}(1-\Phi^{k})}-V'_c(\Phi^{k+1})+V'_e(\Phi^k).
\end{equation}
Here, the map from $\Phi^k$ to $\Phi^{k+1}$ is an optimal transport in the sense that $\Phi^{k+1}$ minimizes the functional
\begin{equation}\label{J}
   J(\Phi):=\frac{1}{2 \tau} \int_{0}^{1}\frac{|\Phi-\Phi^k|^2}{\Phi^k(1-\Phi^k)} d\eta+ W(\Phi),
\end{equation}
where $W$ is a convex functional explicitly given by
$$ W(\Phi)= -\frac{\kappa}{2}\int_{0}^{1} \ln\big(\frac{\partial \Phi}{\partial \eta} \big) d \eta + \frac{\kappa}{2} \int_{0}^{1}\frac{1-2\Phi^k}{\Phi^k(1-\Phi^k)} \Phi d\eta + \int_{0}^{1}V_c (\Phi)d\eta - \int_{0}^{1}V'_e (\Phi^k)\Phi d\eta.$$

Let us introduce the discrete space domain
\begin{equation}\label{Q}
Q:= \{{l_i}: l_{i-1}<l_i, 1\leq i\leq N; \;  l_0=0,l_N=1\}
\end{equation}
and its closure $\bar{Q}:= Q \bigcup \partial Q$ with boundary
\[\partial Q:= \{{l_i} ~| l_{i-1}\leq l_i, 1\leq i\leq N \; \text{and} \; l_{i-1}=l_{i} \; \text{for some} \; 1\leq i\leq N; \; l_0=0, l_N=1\}.\]
The full finite difference discretization of \eqref{semi} is formulated as follows:
\begin{equation}\label{full}
\frac{1}{\Phi^{k}_i(1-\Phi^{k}_i)}\frac{\Phi^{k+1}_i-\Phi^k_i}{\tau} =-\frac{\kappa}{2} \left( \frac{1}{\Phi^{k+1}_{i+1}-\Phi^{k+1}_i}- \frac{1}{\Phi^{k+1}_{i}-\Phi^{k+1}_{i-1}} \right)-\frac{\kappa}{2}\frac{1-2\Phi^{k}_i}{\Phi^{k}_i(1-\Phi^{k}_i)}-V'_c(\Phi^{k+1}_i)+V'_e(\Phi^k_i)
\end{equation}
with the boundary condition $\Phi^{k}_0=0$ and $\Phi^{k}_N=1$ for each $k$.  At each time step, we solve a system of nonlinear equations by a damped Newton's iteration.  For notational simplicity, let us denote the nonlinear functional $F$ in \eqref{full} as
\begin{equation}\label{F}
  F(\Phi^{k+1}_i):=  \frac{1}{\Phi^{k}_i(1-\Phi^{k}_i)}\frac{\Phi^{k+1}_i-\Phi^k_i}{\tau} + \frac{\kappa}{2} \left(\frac{1}{\Phi^{k+1}_{i+1}-\Phi^{k+1}_i}- \frac{1}{\Phi^{k+1}_{i}-\Phi^{k+1}_{i-1}} \right)+\frac{\kappa}{2}\frac{1-2\Phi^{k}_i}{\Phi^{k}_i(1-\Phi^{k}_i)}+V'_c(\Phi^{k+1}_i)-V'_e(\Phi^k_i),
\end{equation}
then we will calculate the Jacobian matrix $DF$ of \eqref{F} and determine the Newton update $\gamma^{k+1,n+1}$ by
\begin{equation}\label{NT}
  DF(\Phi^{k+1,n+1}) \gamma^{k+1,n+1}= -F(\Phi^{k+1,n}),
\end{equation}
where index $k$ corresponds to the temporal discretization and $n$ to the Newton iteration.
Given $\Phi^{k+1,0}=\Phi^k$, $k\in \mathbb{N}$, we calculate
$ \Phi^{k+1,n+1}=\Phi^{k+1,n}+ \alpha (\lambda) \gamma^{k+1,n+1} $
with
\begin{equation*}\label{alpha}
\alpha(\lambda)=\left\{
\begin{array}{ll}
\frac{1}{\lambda}& \text{for} \; \lambda>\lambda', \\
\frac{1-\lambda}{\lambda(3-\lambda)}&\text{for} \; \lambda' \geq \lambda \geq \lambda^*,\\
1&\text{for}  \; \lambda<\lambda^*,\\
\end{array}
\right.
\end{equation*}
where $\lambda^*=2-\sqrt{3}$, $\lambda'\in [\lambda^*,1)$ and
$\lambda(\Phi^{k+1,n})=\sqrt{- h F(\Phi^{k+1,n}) \gamma^{k+1,n}/a_0}$ with $a_0$ defined in \eqref{a0}.

\section{Numerical Analysis}

This section analyzes the implicit Euler scheme \eqref{Euler} and the convex splitting scheme \eqref{full}.  We first collect several important properties of the splitting technique applied in \cite{DLWY2019} that includes the unique solvability, convergence of the Damped Newton method \eqref{NT} and dissipation of energy. Then we prove that there exists a lower bound for the discrete energy, which implies the existence of numerical solutions. Finally, in light of the fact that any steady state must be a Heaviside-type step function, we show the convergence of of the implicit scheme \eqref{Euler}.   

\subsection{Several Properties of the Convex Splitting Scheme}
\begin{lemma*}
The numerical scheme \eqref{full} is uniquely solvable in space $Q$ given by \eqref{Q}.
\end{lemma*}
\begin{proof}
  We first introduce the discretization of \eqref{J} as
  \begin{equation}\label{Jn}
    J_N(y):=\frac{h}{2\tau} \sum_{i=1}^{N-1} \frac{(y_i-\Phi^k_i)^2}{\Phi^k_i(1-\Phi^k_i)}-\frac{\kappa h}{2}\sum_{i=0}^{N-1}\ln \left( \frac{y_{i+1}-y_i}{h} \right)+ \frac{\kappa h}{2}\sum_{i=1}^{N-1} \frac{1-2\Phi^k_i}{\Phi^k_i(1-\Phi^k_i)} y_i  +h\sum_{i=1}^{N-1}V_c(y_i)-h\sum_{i=1}^{N-1}V'_e(\Phi_i^k)y_i
  \end{equation}
with the given $\{\Phi^k_i\}\subset Q$. Since $J_N(y)$ is a convex function on the closed convex set $\bar{Q}$ and $J_N(y)=+\infty$ on boundary $\partial Q$, it is straightforward to obtain that there exists a unique minimizer $x \in Q$.

To show the unique solvability of the scheme, it suffices to prove that $x\in Q$ is the minimizer of $J_N(y)$ if and only if it is a solution of \eqref{full}.  Suppose that $x\in Q$ minimizes $J_N(y)$.  Define $j_1(\epsilon):=J_N(x+\epsilon(y-x))$.  Then there exists $\epsilon_0>0$ small enough such that $x+\epsilon(y-x) \in Q$ for any $(\epsilon,y)\in(-\epsilon_0,\epsilon_0)\times \bar{Q}$.  Now that $j_1(\epsilon)$ achieves its minimizer at zero, one has $j_1'(0)=0$ hence $h\sum_{i=1}^{N-1}(y_i-x_i) F(x_i) =0$ for any $y\in \bar{Q}$.  Therefore, $x$ is a solution of \eqref{full}.

To prove the ``only if" part, we assume that $x\in Q$ solves the scheme \eqref{full} for any $y\in Q$.  Then
  \begin{align*}
    J_N(y) =&J_N(x+(y-x)) \\
    =&J_N(x) +\frac{h}{2\tau} \sum_{i=1}^{N-1} \frac{(y_i-x_i)^2}{\Phi^k_i(1-\Phi^k_i)} +\frac{\kappa h}{2}\sum_{i=0}^{N-1} \left(\frac{y_{i+1}-y_i}{x_{i+1}-x_i} -\ln \big( \frac{y_{i+1}-y_i}{x_{i+1}-x_i} \big) -1 \right)  \\ 
     &+ h\sum_{i=1}^{N-1} \left[ V_c(y_i)-V_c(x_i)-(y_i-x_i)V'_c(x_i) \right]\\
    \geq &  J_N(x),
  \end{align*}
where the last inequality holds since $m-\ln m -1 >0$ for any positive $m$, and $V_c$ is convex function with positive second derivative.  This completes the proof.
\end{proof}

We now introduce the following concept to prove the convergence of the Newton's method of the scheme \eqref{full}.
\begin{defin*}(\cite{Nesterov})
Let $\mathcal{G}$ be  a finite-dimensional real vector space and $\mathcal{Q}$ be an open nonempty convex subset of $\mathcal{G}$.  Then a convex function $\Lambda\in C^3: \mathcal{Q}\rightarrow \mathbb{R}$ is called self-concordant on $\mathcal{Q}$ if there exists a constant $a_0>0$ such that the following inequality holds for all $x\in \mathcal{Q}$ and all $u \in \mathcal{G}$:
   $$ |D^3\Lambda (x)[u,u,u]| \leq 2 a_0^{-1/2} (D^2\Lambda (x)[u,u])^{3/2},$$
where ($D^k\Lambda(x)[u_1,\cdots, u_k]$ is its $k$-th differential taken at $x$ alone the collection of direction $(u_1, \cdots, u_k$).
\end{defin*}
Then we have the following theorem.
\begin{lemma*}
  $J_N(y)$ defined in \eqref{Jn} is a self-concordant function and Newton's iteration \eqref{NT} is convergent in $Q$.
\end{lemma*}
\begin{proof}
  Define $J_N(y):=J_N^1(y)+J_N^2(y)$ with
  \begin{align*}
    J_N^1(y)&:= \frac{h}{2\tau} \sum_{i=1}^{N-1} \frac{(y_i-\Phi^k_i)^2}{\Phi^k_i(1-\Phi^k_i)} +\frac{\kappa h}{2}\sum_{i=1}^{N-1} \frac{1-2\Phi^k_i}{\Phi^k_i(1-\Phi^k_i)} y_i  -h\sum_{i=1}^{N-1}V'_e(\Phi_i^k)y_i, \\
    J_N^2(y)&:= -\frac{\kappa h}{2}\sum_{i=0}^{N-1}\ln \left(\frac{y_{i+1}-y_i}{h} \right) +h\sum_{i=1}^{N-1}V_c(y_i) .
  \end{align*}
Since both the linear and quadratic functions have zero third-order derivatives, one can easily find that $J_N^1(y)$ is self-concordant for any $a_0$ in $Q$.

We proceed to prove that $j_2(\xi):=J_N^2(y+\xi u)$ is a self-concordant function of $\xi$ along every line $u$ in $Q$.  To this end, we have from direct calculations that
\begin{align*}
  j_2''(\xi) =& \frac{h\kappa}{2}\sum_{i=0}^{N-1} \frac{(u_{i+1}-u_{i})^2}{(y_{i+1}+\xi u_{i+1}-y_i-\xi u_i)^2}+h\sum_{i=1}^{N-1}V''_c(y_i+\xi u_i)u_i^2, \\
  j_2'''(\xi) =&-h\kappa  \sum_{i=0}^{N-1} \frac{(u_{i+1}-u_{i})^3}{(y_{i+1}+\xi u_{i+1}-y_i-\xi u_i)^3}+h\sum_{i=1}^{N-1}V'''_c(y_i+\xi u_i)u_i^3.
\end{align*}
Since $V_c(x)$ is convex and smooth, for any $y_i\in Q$ there exists a constant $M_v>0$ such that
\begin{equation}\label{mv}
    |V'''_c(y_i)| \leq M_v (V''_c(y_i))^{\frac{3}{2}}.
\end{equation}
Let us define 
\begin{equation*}
  m_i:=\left\{
  \begin{array}{cc}
    \left| \frac{u_{i+1}-u_i}{y_{i+1}+\xi u_{i+1}-y_i-\xi u_i} \right|, & i=0,1,\cdots,N-1, \\
    (V''_c(y_{i-N+1}+\xi u_{i-N+1}))^{\frac{1}{2}}|u_{i-N+1}|, & i=N,\cdots,2N-2,
  \end{array}
  \right.
\end{equation*}
then we proceed to find that
\begin{align*}
  |j_2'''(\xi)| \leq& \,h\kappa  \sum_{i=0}^{N-1} \frac{|u_{i+1}-u_{i}|^3}{|y_{i+1}+\xi u_{i+1}-y_i-\xi u_i|^3}+h\sum_{i=1}^{N-1}|V'''_c(y_i+\xi u_i)| |u_i|^3 \\
   \leq& h\kappa  \sum_{i=0}^{N-1} \frac{|u_{i+1}-u_{i}|^3}{|y_{i+1}+\xi u_{i+1}-y_i-\xi u_i|^3}+h M_v \sum_{i=1}^{N-1}(V''_c(y_i+\xi u_i))^{\frac{3}{2}} |u_i|^3 \\
   \leq&\, h \max(\kappa,M_v) \sum_{i=0}^{2N-2} m_i^3 
   \leq h \max(\kappa,M_v) \left(\sum_{i=0}^{2N-2} m_i^2 \right)^{\frac{3}{2}}
   \leq \frac{\max(\kappa,M_v)}{\sqrt{h} (\min(1,\frac{\kappa}{2}))^{\frac{3}{2}}} \left( j_2''(\xi) \right)^{\frac{3}{2}},
\end{align*}
where the fourth inequality follows from the following
\begin{equation*}
  \left| \sum_{i=0}^{2N-2} m_i^3 \right|\leq \left(\sum_{i=0}^{2N-2} m_i^2\right)^{\frac{1}{2}} \left(\sum_{i=0}^{2N-2} m_i^4\right)^{\frac{1}{2}} \leq \left(\sum_{i=0}^{2N-2} m_i^2 \right)^{\frac{3}{2}},
\end{equation*}

Now, let us choose
\begin{equation}\label{a0}
a_0:=\frac{4h (\min(1,\frac{\kappa}{2}))^3}{(\max(\kappa,M_v))^2}
\end{equation}
with $M_v$ satisfies \eqref{mv},
then $J^2_N(y)$ is self-concordant for $y\in Q$ with parameter $a_0$.  This implies that $J_N(y)$ is self-concordant and the Newton's iteration is convergent thanks to Theorem 2.2.3 in \cite{Nesterov}.
\end{proof}

 Let $\{\Phi_i\}_{i=0,1,...,N}$ be any strictly increasing sequence with $\Phi_0=0$ and $\Phi_N=1$.  Then the discrete free energy functional is defined as
\begin{equation}\label{discreteenergy}
  \mathcal{E}_N(\Phi)= -\frac{\kappa}{2}\sum_{i=0}^{N-1} \ln \left(\frac{\Phi_{i+1}-\Phi_{i}}{h} \right)h +\frac{\kappa}{2}\sum_{i=1}^{N-1} \ln \left(\Phi_{i}(1-\Phi_{i}) \right) h+\sum_{i=1}^{N-1}V(\Phi_i)h.
\end{equation}
Then the following theorem states that the energy dissipation is preserved through the discretization.
\begin{lemma*}
The discrete energy dissipation for the evolution of the discrete energy \eqref{discreteenergy}
\begin{equation}\label{energydis}
 \mathcal{E}_N(\Phi^{k+1})-\mathcal{E}_N(\Phi^{k}) + \sum_{i=1}^{N-1} \frac{\left(\Phi^{k+1}_i-\Phi^k_i \right)^2}{\Phi^{k}_i(1-\Phi^{k}_i)}
    \frac{h}{\tau}
  \leq  0,
\end{equation}
holds for the scheme \eqref{full}.
\end{lemma*}
\begin{proof}
We rewrite
\[\mathcal{E}_N(\Phi^{k+1}):= \mathcal{E}_N^1(\Phi^{k+1})-\mathcal{E}_N^2(\Phi^{k+1}),\]
with
 $$\mathcal{E}_N^1(\Phi^{k+1}):=-\frac{\kappa}{2}\sum_{i=0}^{N-1} \ln \left( \frac{\Phi^{k+1}_{i+1}-\Phi^{k+1}_{i}}{h} \right)h +\sum_{i=1}^{N-1}V_c(\Phi^{k+1}_i)h$$
  and
  $$\mathcal{E}_N^2(\Phi^{k+1}):=-\frac{\kappa}{2}\sum_{i=1}^{N-1} \ln \left( \Phi^{k+1}_{i}(1-\Phi^{k+1}_{i}) \right) h+\sum_{i=1}^{N-1}V_e(\Phi^{k+1}_i)h.$$
  Define
    \begin{align*}
    &\frac{\delta \mathcal{E}_N^1}{\delta \Phi^{k+1}}(\Phi^{k+1}_i) =\frac{\kappa}{2}\left( \frac{1}{\Phi^{k+1}_{i+1}-\Phi^{k+1}_i}- \frac{1}{\Phi^{k+1}_{i}-\Phi^{k+1}_{i-1}} \right)+V'_c(\Phi_i^{k+1}), \\
    &\frac{\delta \mathcal{E}_N^2}{\delta \Phi^{k+1}}(\Phi^{k}_i) =- \frac{\kappa}{2}\frac{1-2\Phi^{k}_i}{\Phi^{k}_i(1-\Phi^{k}_i)}  +V'_e(\Phi_i^k).
  \end{align*}
According to the inequality $\ln (\frac{x}{y})\leq (x-y)\frac{1}{y}$ and convexity of $V_c$, we can obtain
$$\ln \left(\frac{\Phi_{i+1}^{k+1}-\Phi_{i}^{k+1}}{\Phi_{i+1}^{k}-\Phi_{i}^{k}} \right)\geq -(\Phi_{i+1}^{k}-\Phi_{i}^{k}-\Phi_{i+1}^{k+1}+\Phi_{i}^{k+1})\frac{1}{\Phi_{i+1}^{k+1}-\Phi_{i}^{k+1}},$$
and
$$\sum_{i=1}^{N-1}V_c(\Phi^{k}_i)h -\sum_{i=1}^{N-1}V_c(\Phi^{k+1}_i)h\geq \sum_{i=1}^{N-1}(\Phi_i^k-\Phi_i^{k+1})V'_c(\Phi^{k+1}_i)h.$$
Then it follows from the summation-by-parts that
\begin{align}\label{en1}
  \mathcal{E}_N^1(\Phi^{k}) - \mathcal{E}_N^1(\Phi^{k+1}) 
  =&\, \frac{h\kappa}{2} \sum_{i=0}^{N-1} \ln \left( \frac{\Phi^{k+1}_{i+1}-\Phi^{k+1}_{i}}{\Phi^{k}_{i+1}-\Phi^{k}_{i}} \right)+\sum_{i=1}^{N-1}V_c(\Phi^{k}_i)h -\sum_{i=1}^{N-1}V_c(\Phi^{k+1}_i)h \notag\\
   \geq&\,  -\frac{h\kappa}{2} \sum_{i=0}^{N-1} (\Phi_{i+1}^{k}-\Phi_{i}^{k}-\Phi_{i+1}^{k+1}+\Phi_{i}^{k+1})\frac{1}{\Phi_{i+1}^{k+1}-\Phi_{i}^{k+1}}+\sum_{i=1}^{N-1}(\Phi_i^k-\Phi_i^{k+1})V'_c(\Phi^{k+1}_i)h \notag\\
   =&\, \frac{h\kappa}{2}\sum_{i=1}^{N-1} (\Phi_{i}^{k}-\Phi_{i}^{k+1})\left( \frac{1}{\Phi^{k+1}_{i+1}-\Phi^{k+1}_i}- \frac{1}{\Phi^{k+1}_{i}-\Phi^{k+1}_{i-1}} \right)+\sum_{i=1}^{N-1}(\Phi_i^k-\Phi_i^{k+1})V'_c(\Phi^{k+1}_i)h \notag\\
   =&\,h\sum_{i=1}^{N-1}(\Phi^{k}_i-\Phi^{k+1}_i)\frac{\delta \mathcal{E}_N^1}{\delta \Phi^{k+1}}(\Phi^{k+1}_i). 
\end{align}
  
Similarly, in light of the convexity of $\mathcal{E}_N^2$ we find that
  \begin{equation*}
     \mathcal{E}_N^2(\Phi^{k+1})-\mathcal{E}_N^2(\Phi^{k}) \geq h\sum_{i=1}^{N-1}(\Phi^{k+1}_i-\Phi^{k}_i)\frac{\delta \mathcal{E}_N^2}{\delta \Phi^{k+1}}(\Phi^{k}_i).
  \end{equation*}
  
On the other hand, let us rewrite scheme \eqref{full} into the following
\begin{align}\label{p1}
   \frac{1}{\Phi^{k}_i(1-\Phi^{k}_i)}\frac{\Phi^{k+1}_i-\Phi^k_i}{\tau}=& -\frac{\delta \mathcal{E}_N^1}{\delta \Phi^{k+1}}(\Phi^{k+1}_i)+\frac{\delta \mathcal{E}_N^2}{\delta \Phi^{k+1}}(\Phi^{k}_i),
\end{align}
then one infers from \eqref{en1}-\eqref{p1} that
  \begin{align*}
 \mathcal{E}_N(\Phi^{k+1})-\mathcal{E}_N(\Phi^{k})
 =&\,\mathcal{E}_N^1(\Phi^{k+1})-\mathcal{E}_N^2(\Phi^{k+1})-\mathcal{E}_N^1(\Phi^{k})+\mathcal{E}_N^2(\Phi^{k})\\
 \leq&\, \sum_{i=1}^{N-1} \left(\frac{\delta \mathcal{E}_N^1}{\delta \Phi^{k+1}}(\Phi^{k+1}_i)-\frac{\delta \mathcal{E}_N^2}{\delta \Phi^{k+1}}(\Phi^{k}_i)\right) \left(\Phi^{k+1}_i-\Phi^k_i \right) h\\ 
  =&\,-\sum_{i=1}^{N-1} \frac{\left(\Phi^{k+1}_i-\Phi^k_i \right)^2}{\Phi^{k}_i(1-\Phi^{k}_i)}
    \frac{h}{\tau}
  \leq  0,
\end{align*}
which is claimed.
\end{proof}

\subsection{Convergence of the Euler Implicit Scheme}
We proceed to study convergence of the discrete solution of \eqref{Euler} to a unique steady state in the long time.  For this purpose we first prove the existence of this steady state. and our argument begins with the following observation.
\begin{lemma}\label{lemma35}
  Let $\{\Phi_i\}_{i=0,1,...,N}$ be any strictly increasing sequence with $\Phi_0=0$ and $\Phi_N=1$.  Then the discrete free energy functional $\mathcal{E}_N(\Phi)$ \eqref{discreteenergy} 
  is bounded from below for each $h>0$. Moreover, the bound is of the order $O(\ln h)$ as $h$ goes to zero.
\end{lemma}
\begin{proof}
The trick of the proof is to find the ``median" of the sequence.  Let $k^0$ be the integer such that
\[\Phi_i \in (0,\frac{1}{2}) ~\text{for}~i\leq k^0,\quad\text{and}\quad\Phi_i \in [\frac{1}{2}, 1) ~\text{for}~i> k^0.\]
Therefore, one finds that
\[  \frac{\Phi_{i+1}-\Phi_{i}}{\Phi_{i+1}(1-\Phi_{i+1})} = \frac{1}{1-\Phi_{i+1}}-\frac{\Phi_{i}}{\Phi_{i+1}(1-\Phi_{i+1})}
   \overbrace{\leq 2- \frac{\Phi_i}{\Phi_{i+1}}}^{\text{since} \; 0< \Phi_{i+1}<\frac{1}{2}}
   \leq 2,\quad \text{for~}i=0,1,\cdots, k^0-1,\]
and 
\[  \frac{\Phi_{i+1}-\Phi_{i}}{\Phi_{i}(1-\Phi_{i})} \leq \overbrace{\frac{1}{\Phi_{i}} \leq 2}^{\text{since~}\Phi_{i}\geq \frac{1}{2}},\quad \text{for~}i=k^0+1, \cdots, N-1.\]
These inequalities, together with the uniform boundedness $|V(x)|\leq M$ in $[0,1]$, enable us to estimate
\begin{align*}
  \mathcal{E}_N(\Phi)
  \geq&\,  \frac{\kappa}{2}(1-h) \ln h -\frac{\kappa}{2}\sum_{i=0}^{N-1} \ln \left( \Phi_{i+1}-\Phi_{i} \right)h +\frac{\kappa}{2}\sum_{i=1}^{N-1} \ln \left(\Phi_{i}(1-\Phi_{i}) \right) h -M(N-1)h \\
   =&\, \frac{\kappa}{2}(1-h) \ln h -h \ln (\Phi_{k^0+1}-\Phi_{k^0}) -\frac{h\kappa}{2}\sum_{i=0}^{k^0-1} \ln \left(\frac{\Phi_{i+1}-\Phi_{i}}{\Phi_{i+1}(1-\Phi_{i+1})} \right) \\ &\,-\frac{h\kappa}{2}\sum_{i=k^0+1}^{N-1} \ln \left( \frac{\Phi_{i+1}-\Phi_{i}}{\Phi_{i}(1-\Phi_{i})} \right)-M+Mh
\end{align*}
\begin{align*}
   \geq&\, \frac{\kappa}{2}(1-h) \ln h -\frac{\kappa}{2}h(N-1)\sum_{i=0}^{k^0-1} \frac{1}{N-1}\ln \left(\frac{\Phi_{i+1}-\Phi_{i}}{\Phi_{i+1}(1-\Phi_{i+1})} \right) \\ &\,-\frac{\kappa}{2}h(N-1)\sum_{i=k^0+1}^{N-1} \frac{1}{N-1} \ln \left( \frac{\Phi_{i+1}-\Phi_{i}}{\Phi_{i}(1-\Phi_{i})} \right)-M+Mh  \\
   \geq&\,  \frac{\kappa}{2}(1-h)\ln h  -M+Mh \\
   &-\frac{\kappa}{2}h (N-1) \ln \left[ \frac{1}{N-1} \left( \sum_{i=0}^{k^0-1} \frac{\Phi_{i+1}-\Phi_{i}}{\Phi_{i+1}(1-\Phi_{i+1})} +\sum_{i=k^0+1}^{N-1} \frac{\Phi_{i+1}-\Phi_{i}}{\Phi_{i}(1-\Phi_{i})} \right) \right]\\
   \geq&\,  \frac{\kappa}{2}(1-h)\ln h -M+Mh-\frac{\kappa}{2}(1-2h)\ln 2,
\end{align*}
where we apply the fact $0<\Phi_{k^0+1}-\Phi_{k^0}<1$ for the second inequality and the Jensen inequality for the third one.
\end{proof}

As a consequence we establish the existence of a stationary state realized as a minimizer of the free energy functional.  Notice that lower bound of the discrete energy functional \eqref{energydis} diverges to $-\infty$ as $h\to 0$, and this is consistent with the fact that the continuous energy functional \eqref{freeenergy} is not bounded from below.

Now, we are ready to show that the steady state of \eqref{phi} must be a Heaviside step function.
\begin{lemma}\label{lemma_step1}
Denote $\{\Phi_i\}_{i=0,1,...,N}$ as the discrete solution of \eqref{Euler}. Define $\{\mathcal{F}_i^k\}_{i=0,1,\cdots,N}^{k\in\mathbb{N}}$ as
\begin{equation}\label{mathf}
  \mathcal{F}_i^k:=\tau \sum_{j=1}^{k}\Phi_i^j (1-\Phi_i^j)V'(\Phi_i^j), \quad \text{with~}\mathcal{F}_i^0=0.
\end{equation}
Then for any fix time step size $\tau>0$, the following statements hold:
\begin{enumerate}[(i)]
  \item the conservation law holds for the scheme \eqref{Euler} as follows
 \begin{equation}\label{conserdiscrete}
   h\sum_{i=1}^{N-1} \left(\Phi_i^{k+1} +\mathcal{F}_i^{k+1} \right)=h\sum_{i=1}^{N-1} \left(\Phi_i^{k} +\mathcal{F}_i^{k} \right).
 \end{equation}
  \item  let $\{\mathcal{F}_i^*\}_{i=0,1,...,N}$ be the limit of $\{\mathcal{F}_i^k\}_{i=0,1,\cdots,N}^{k\in\mathbb{N}}$.
 Introduce the sequence $\{\Phi_i^*\}_{i=0,1,...,N}$
\begin{equation}\label{phistar}
\Phi_i^*=\left\{
\begin{array}{ll}
0,& \text{for} \; i=0,1,\cdots,m, \\
a,& \text{for} \; i=m+1, \\
1,&\text{for} \; i=m+2,\cdots,N,\\
\end{array}
\right.
\end{equation}
where parameters $a \in [0,1]$ and $m\in \{0,1,\cdots,N-1\}$ are determined by
\begin{equation}\label{am}
a(1-a)V'(a)=0, \quad\text{and}\quad
\sum_{i=1}^{N-1} \left(\Phi_i^{*} +\mathcal{F}_i^{*} \right)=\sum_{i=1}^{N-1} \Phi_i^{0}.
\end{equation}
 If $V'(\cdot)$ does not change sign in $[0,1]$, the steady state of the numerical scheme \eqref{Euler} must be $\{\Phi_i^*\}_{i=0,1,...,N}$.
\end{enumerate}
\end{lemma}

\begin{proof}
Summing over $i$, one can find from \eqref{Euler} that
\begin{align*}
  &h\sum_{i=1}^{N-1} \left(\Phi_i^{k+1} +\mathcal{F}_i^{k+1} \right)-\,\,h\sum_{i=1}^{N-1} \left(\Phi_i^{k} +\mathcal{F}_i^{k} \right)\\ =&\, h\sum_{i=1}^{N-1} \left(\Phi_i^{k+1}-\Phi_i^{k}  \right)+\tau h \sum_{i=1}^{N-1}  \Phi_i^{k+1} (1-\Phi_i^{k+1})V'(\Phi_i^{k+1}) \\
   =&\,-\tau h \sum_{i=1}^{N-1}\Phi^{k+1}_i(1-\Phi^{k+1}_i)\left(\frac{1}{\Phi^{k+1}_{i+1}-\Phi^{k+1}_i}- \frac{1}{\Phi^{k+1}_{i}-\Phi^{k+1}_{i-1}} \right)-\tau h \sum_{i=1}^{N-1}(1-2\Phi^k_i)  \\
      =&\, h\sum_{i=1}^{N-2}(1-\Phi^{k+1}_i-\Phi^{k+1}_{i+1})+ h(1-\Phi_1^{k+1})-h\Phi_{N-1}^{k+1}
   -h\sum_{i=1}^{N-1}(1-2\Phi^{k+1}_i)
   =0,
\end{align*}
which implies that the conservation law in (i) holds for the scheme.

Suppose that $V'(\cdot)$ is of one sign in $[0,1]$.  Let $\{\Phi_i^*\}_{i=0,1,...,N}$ be an equilibrium of scheme \eqref{Euler} such that
 \begin{equation}\label{eq1}
   0=- \frac{\kappa}{2}\left(\frac{1}{\Phi^{*}_{i+1}-\Phi^{*}_{i}}- \frac{1}{\Phi^{*}_{i}-\Phi^{*}_{i-1}} \right)-\frac{\kappa}{2}\frac{1-2\Phi^{*}_{i}}{\Phi^{*}_{i}(1-\Phi^{*}_{i})}-V'(\Phi^*_{i}).
 \end{equation}
Let $m_1$ be the largest spatial index for $\Phi_{m_1}^*=0$ and $m_2$ be the smallest for $\Phi_{m_2+2}^*=1$.  Then \eqref{eq1} implies
\begin{align}
   &0=- \frac{\kappa}{2}\left(\frac{1}{\Phi^{*}_{m_1+2}-\Phi^{*}_{m_1+1}}- \frac{1}{1-\Phi^{*}_{m_1+1}} \right)-V'(\Phi^*_{m_1+1}), \label{m1} \\
   &0=- \frac{\kappa}{2}\left(\frac{1}{\Phi^{*}_{m_2+1}}- \frac{1}{\Phi^{*}_{m_2+1}-\Phi^{*}_{m_2}} \right)-V'(\Phi^*_{m_2+1}), \label{m2}
\end{align}
hence $V'(\Phi^*_{m_1+1})\leq0$ and $V'(\Phi^*_{m_2+1})\geq0$.  However, $V'(\Phi^*_{m_1+1})\leq0$ leads a contradiction to $V'(\cdot)>0$ since we then must have $\Phi^*_{i}\equiv 1$ for $i=m_1+1,...,N$.  One also gets a contradiction if $V'(\cdot)<0$.  Therefore, the definitions of $m_1$ and $m_2$ imply that in either case the steady state is a nondecreasing step function from zero to one, whereas the jump location is determined by the conservation law \eqref{conserdiscrete}.  

Now that $\{\mathcal{F}_i^k\}_{i=0,1,\cdots,N}^{k\in\mathbb{N}}$ is bounded in time thanks to \eqref{mathf}, the Bolzano--Weierstrass theorem finds a subsequence convergence to $\{\mathcal{F}_i^*\}_{i=0,1,...,N}$ and then the steady state can be written as \eqref{phistar} with $a=1$ and $m$  determined by \eqref{am}.  Furthermore, for any $\epsilon>0$, we can choose $\tau$ small enough such that for all $n_2>n_1>1$
\begin{equation*}
  |\mathcal{F}_i^{n_2}-\mathcal{F}_i^{n_1}|  =\left|  \tau\sum_{j=n_1+1}^{n_2}\Phi_i^j (1-\Phi_i^j)V'(\Phi_i^j) \right|
    \leq \frac{\tau (n_2-n_1)\max_{[0,1]}|V'(x)|}{4}\leq\epsilon,
\end{equation*}
which implies that $\{\mathcal{F}_i^k\}_{i=0,1,\cdots,N}^{k\in\mathbb{N}}$ is Cauchy sequence with respect to $k$. Therefore, as $\tau$ goes to zero, we have the unique existence of $\{\mathcal{F}_i^*\}_{i=0,1,...,N}$, which implies the existence and uniqueness of the steady state as expected.
\end{proof}

\begin{remark}
For diffusion equation (\ref{diff}), \eqref{m1} and \eqref{m2} imply that $m_1=m_2$ as $V'\equiv0$, hence there exists $a \in (0,1)$ such that the steady state is given by \eqref{phistar}.  Moreover, the conservation of mass center within the discrete scheme implies that
\begin{equation}\label{amdrift}
a=\sum_{i=0}^{N} \Phi^*_i -\Big[\sum_{i=0}^{N} \Phi^0_i\Big], \quad
m=N-1-\Big[\sum_{i=0}^{N} \Phi^0_i\Big],
 \end{equation}
where $[ \cdot ]$ it the integer-valued function. The assumption $V'(\cdot)$ being one sign is technical and our numerics suggest that the conclusion still holds if otherwise. 
\end{remark}

\begin{lemma}\label{lemma_step2}
Let $\{\Phi_i\}_{i=0,1,...,N}$ be the discrete solution of \eqref{Euler}.  Define $\{\mathcal{F}_i^k\}_{i=0,1,\cdots,N}^{k\in\mathbb{N}}$ as follows
\begin{equation*}
  \mathcal{F}_i^k:=\tau \sum_{j=1}^{k}\Phi_i^j (1-\Phi_i^j)V'(\Phi_i^j), \quad \mathcal{F}_i^0=0.
\end{equation*}
Then for any smooth function $V(\cdot)$ the following statements are true:
\begin{enumerate}[(i)]
  \item the conservation law \eqref{conserdiscrete} holds for the scheme \eqref{Euler};
  \item let $\{\mathcal{F}_i^*\}_{i=0,1,...,N}$ be the limit of  $\{\mathcal{F}_i^k\}_{i=0,1,\cdots,N}^{k\in\mathbb{N}}$, then the sequence $\{\Phi_i^*\}_{i=0,1,...,N}$ in \eqref{phistar} supplemented by \eqref{am} is a steady state of the numerical scheme \eqref{Euler}.
\end{enumerate}
\end{lemma}

\begin{proof}
  The proof of the first statement is quite similar to that for \cref{lemma_step1} and it is omitted.
  To show that \eqref{phistar} under \eqref{am} is a steady state of \eqref{Euler}, we first recall that the discrete function $\{\Phi_i^*\}_{i=0,1,...,N}$ is an equilibrium if and only if $\Phi_i^*=0$ or $\Phi_i^*=1$ or $\Phi_i^*$ satisfies 
   \begin{equation*}
   0=- \frac{\kappa}{2}\left(\frac{1}{\Phi^{*}_{i+1}-\Phi^{*}_{i}}- \frac{1}{\Phi^{*}_{i}-\Phi^{*}_{i-1}} \right)-\frac{\kappa}{2}\frac{1-2\Phi^{*}_{i}}{\Phi^{*}_{i}(1-\Phi^{*}_{i})}-V'(\Phi^*_{i}).
 \end{equation*}
 In order to show that \eqref{phistar} supplemented by \eqref{am} is a steady state of the scheme \eqref{Euler}, one only needs to verify the condition of $\Phi_{m+1}^*$. If $V'(\cdot)$ is strictly positive or negative, \eqref{am} implies that $a=0$ or $a=1$. If $V'(\cdot)=0$, the steady state is uniquely determined by \eqref{amdrift}.  If $V'(\cdot)$ changes sign in $[0,1]$, then \eqref{Euler} implies
 \begin{equation}\label{eq2}
   0=- \frac{\kappa}{2}\left(\frac{1}{\Phi^{*}_{m+2}-\Phi^{*}_{m+1}}- \frac{1}{\Phi^{*}_{m+1}-\Phi^{*}_{m}} \right)-\frac{\kappa}{2}\frac{1-2\Phi^{*}_{m+1}}{\Phi^{*}_{m+1}(1-\Phi^{*}_{m+1})}-V'(\Phi^*_{m+1}).
 \end{equation}
Substituting $\Phi_m^*=0$, $\Phi_{m+1}^*=a$ and $\Phi_{m+2}^*=1$ into \eqref{eq2} gives that $V'(a)=0$, which is expected.
If $\{\mathcal{F}_i^k\}_{i=0,1,\cdots,N}^{k\in\mathbb{N}}$ has the limit $\{\mathcal{F}_i^*\}_{i=0,1,...,N}$, the parameter $m$ can be determined by \eqref{am}.  In either case, we prove the statement in (ii).  
\end{proof}

We next analyze the long time properties of the numerical scheme concerning its convergence to the steady state.  We will first show the exponential decay for the diffusion equation \eqref{diff} for any fixed time step size modulo error terms.  As for the replicator-diffusion equation \eqref{driftdiff} we will be able to obtain a similar result in the case of monotone drift fitness potentials.
\begin{theorem}\label{theorem_decay1} 
The solution of the numerical scheme \eqref{Euler} for the diffusion equation \eqref{diff} satisfies
$$
\|\Phi^k-\Phi^*\|^2 \leq \|\Phi^0-\Phi^*\|^2 \left(\frac{1}{1+2\tau}\right)^k+ O(h) \simeq C_0 \exp{(-2k\tau)}+ O(h) \,,
$$
for $h$ small enough and $C_0$ both depending only on the distance of the initial data to the steady state $\|\Phi^0-\Phi^*\|$, that is, the fully discrete numerical scheme converges exponentially to the steady state in time modulo $O(h)$ terms.
\end{theorem}
\begin{proof}
To show the convergence to the steady state, we recall from \cref{lemma_step1} that the discrete solution takes the form
    \begin{equation*}
\Phi_i^*=\left\{
\begin{array}{ll}
0,& \text{for} \; i=0,1,\cdots,m, \\
a,& \text{for} \; i=m+1, \\
1,&\text{for} \; i=m+2,\cdots,N,\\
\end{array}
\right.
\end{equation*}
with $a$ and $m$ to be determined.  The time evolution of the $L^2$-distance between $\Phi^k$ and the stationary state $\Phi^*$ reads
\begin{align}\label{L2}
& \frac{1}{2\tau} \left(\|\Phi^{k+1}-\Phi^* \|^2 - \|\Phi^{k}-\Phi^* \|^2 \right)\\
=&\,\frac{h}{2\tau} \sum_{i=1}^{m}\left[(\Phi^{k+1}_i)^2-(\Phi^{k}_i)^2 \right]+ \frac{h}{2\tau}\left[ (\Phi^{k+1}_{m+1}-a)^2 -(\Phi^{k}_{m+1}-a)^2 \right]  \nonumber + \frac{h}{2\tau} \sum_{i=m+2}^{N-1}\left[(\Phi^{k+1}_i-1)^2-(\Phi^{k}_i-1)^2 \right]  \nonumber\\
=\,&\frac{h}{2}\sum_{i=1}^{N-1}\frac{\Phi^{k+1}_i-\Phi^{k}_i}{\tau}(\Phi^{k+1}_i+\Phi^{k}_i) -ah \frac{\Phi^{k+1}_{m+1}-\Phi^{k}_{m+1}}{\tau} -h\sum_{i=m+2}^{N-1}\frac{\Phi^{k+1}_i-\Phi^{k}_i}{\tau}  \nonumber\\
=\,&h\sum_{i=1}^{N-1}\frac{\Phi^{k+1}_i-\Phi^{k}_i}{\tau}\Phi^{k+1}_i -\frac{h}{2}\sum_{i=1}^{N-1}\frac{(\Phi^{k+1}_i-\Phi^{k}_i)^2}{\tau} -ah \frac{\Phi^{k+1}_{m+1}-\Phi^{k}_{m+1}}{\tau}-h\sum_{i=m+2}^{N-1}\frac{\Phi^{k+1}_i-\Phi^{k}_i}{\tau}  \nonumber\\
\leq\,&\overbrace{h\sum_{i=1}^{N-1}\frac{\Phi^{k+1}_i-\Phi^{k}_i}{\tau}\Phi^{k+1}_i}^{A_n}
      \overbrace{-ah \frac{\Phi^{k+1}_{m+1}-\Phi^{k}_{m+1}}{\tau}}^{B_n}  \overbrace{-h\sum_{i=m+2}^{N-1}\frac{\Phi^{k+1}_i-\Phi^{k}_i}{\tau}}^{C_n}. \nonumber
\end{align}
To estimate $A_n,B_n$ and $C_n$, we apply \eqref{Euler} with $V'(\cdot)=0$ and $\kappa=2$ and deduce from the summation-by-part that
 \begin{align*}
   A_n =&-h\sum_{i=1}^{N-1} (\Phi_i)^2(1-\Phi_i)\left(\frac{1}{\Phi_{i+1}-\Phi_i}- \frac{1}{\Phi_{i}-\Phi_{i-1}} \right)-h\sum_{i=1}^{N-1}\Phi_i(1-2\Phi_i) \\
    =&\,h\sum_{i=0}^{N-1} \left[(\Phi_{i+1})^2(1-\Phi_{i+1})-(\Phi_i)^2(1-\Phi_i) \right] \frac{1}{\Phi_{i+1}-\Phi_i}-h\sum_{i=1}^{N-1}\Phi_i(1-2\Phi_i)  \\
    =&\,h\sum_{i=0}^{N-1} \left[\Phi_{i+1}+\Phi_{i}-(\Phi_{i+1})^2-\Phi_{i+1}\Phi_{i}-(\Phi_i)^2 \right] -h\sum_{i=0}^{N-1}\Phi_i(1-2\Phi_i)   \\
    =&\,h\sum_{i=0}^{N-1}  \Phi_{i+1}(1-\Phi_{i})-h 
    = h\sum_{i=1}^{N-1}  \Phi_{i}-h\sum_{i=0}^{N-1} ( \Phi_{i+1}-\Phi_i)\Phi_{i}-h\sum_{i=0}^{N-1} (\Phi_i)^2\\
   \leq& h\sum_{i=1}^{N-1}  \Phi_{i}-h\sum_{i=0}^{N-1} (\Phi_i)^2,
 \end{align*}
\begin{align*}
  B_n &= ah \Phi_{m+1}(1-\Phi_{m+1})\left(\frac{1}{\Phi_{m+2}-\Phi_{m+1}}- \frac{1}{\Phi_{m+1}-\Phi_{m}} \right)+ah(1-2\Phi_{m+1}) 
   \\&\leq ah \frac{\Phi_{m+1}(1-\Phi_{m+1})}{\Phi_{m+2}-\Phi_{m+1}}+ah(1-2\Phi_{m+1})\,,
\end{align*}
and
 \begin{align*}
   C_n =&\,h\sum_{i=m+2}^{N-1} \Phi_i(1-\Phi_i)\left(\frac{1}{\Phi_{i+1}-\Phi_i}- \frac{1}{\Phi_{i}-\Phi_{i-1}} \right)+h\sum_{i=m+2}^{N-1}(1-2\Phi_i) \\
    =&\,-h\sum_{i=m+2}^{N-1} (1-\Phi_{i+1}-\Phi_{i}) -h \frac{\Phi_{m+2}(1-\Phi_{m+2})}{\Phi_{m+2}-\Phi_{m+1}} +h\sum_{i=m+2}^{N-1}(1-2\Phi_i) \\
    =&\,-h \frac{\Phi_{m+1}(1-\Phi_{m+2})}{\Phi_{m+2}-\Phi_{m+1}}
    \leq -ah\frac{\Phi_{m+1}(1-\Phi_{m+2})}{\Phi_{m+2}-\Phi_{m+1}},
 \end{align*}
where we skip the index $k+1$ for simplicity.  On the other hand, the conservation of mass center implies that $\sum_{i=1}^{N-1}\Phi_i^k=N+a-m-2$ for any $k\in \mathbb{N}$, then
\begin{align*}
  &\frac{1}{2\tau} \left(\|\Phi^{k+1}-\Phi^* \|^2 - \|\Phi^{k}-\Phi^* \|^2 \right)\\ 
  \leq&\, h\sum_{i=1}^{N-1}  \Phi_{i}-h\sum_{i=0}^{N-1} (\Phi_i)^2 + ah \frac{\Phi_{m+1}(\Phi_{m+2}-\Phi_{m+1})}{\Phi_{m+2}-\Phi_{m+1}} +ah(1-2\Phi_{m+1})\\
  =&\,h\sum_{i=1}^{N-1}  \Phi_{i}+ah-ah \Phi_{m+1}-h\sum_{i=1}^{N-1} (\Phi_i)^2\\
  =&\,-\left[ h\sum_{i=1}^{m} (\Phi_i)^2 + h(\Phi_{m+1}-a)^2 +h\sum_{i=m+2}^{N-1} (\Phi_i-1)^2 \right]+2h\sum_{i=1}^{m+1}  \Phi_{i}-3ah \Phi_{m+1}+a^2h\\
  \leq&\, -\|\Phi^{k+1}-\Phi^* \|^2+2h\sum_{i=1}^{m+1}  \Phi_{i}+h.
\end{align*}
Denote $b_{k+1}:=\sum_{i=1}^{m+1}\Phi^{k+1}_i$, then accumulating \eqref{Euler} with respect to $i$ yields
$$b_{k+1}-b_k=-\tau \frac{\Phi_{m+1}^{k+1}(1-\Phi_{m+2}^{k+1})}{\Phi^{k+1}_{m+2}-\Phi_{m+1}^{k+1}}<0.$$
Since $0\leq b_k \leq m+1$ for each $k\in \mathbb{N}$, $\{b_k\}$ is a Cauchy sequence and
$\lim_{k\rightarrow \infty}b_k=\sum_{i=1}^{m+1}\Phi^*_i=a\leq 1.$
Therefore, for any fixed $\tau>0$, there exists some $k_0 \in\mathbb{N}$ such that $b_{k+1}\leq 2$ for any $k>k_0$, and
\begin{equation}\label{exp1}
  \frac{1}{2\tau} \left(\|\Phi^{k+1}-\Phi^* \|^2 - \|\Phi^{k}-\Phi^* \|^2 \right)\leq -\|\Phi^{k+1}-\Phi^* \|^2+5h, \forall k>k_0,
\end{equation}
which implies that
\begin{equation*}
\|\Phi^k-\Phi^*\|^2 \leq \left(\frac{1}{1+2\tau}\right)^k \|\Phi^0-\Phi^*\|^2  +5h \left[ 1-\left(\frac{1}{1+2\tau}\right)^k \right]
=\left(\|\Phi^0-\Phi^*\|^2-5h \right) \left(\frac{1}{1+2\tau}\right)^k+5h.
\end{equation*}

Finally, let us choose the space step small such that $h\leq \|\Phi^0-\Phi^*\|^2/5$.  Then one concludes that $\left( \frac{1}{1+2\tau}\right)^k \approx \exp (-2k\tau)\approx \exp(-2t)$ for $\tau$ small, and this completes the proof.
\end{proof}

\begin{remark}
   When $h$ goes to zero, we can deduce a posteriori uniqueness of the equilibrium from the convergence property.  Indeed, let  $\tilde{\Phi}^*$ be another equilibrium state and set $\Phi^{k}=\Phi^{k+1}=\tilde{\Phi}^*$ in \eqref{exp1} with $h=0$.   Then we have from the evolution that $\|\tilde{\Phi}^*-\Phi^*\|^2 \leq 0$ which proves the uniqueness of the stationary solution.
\end{remark}

\begin{remark}
In the continuum case one can linearize the equation around the steady states (i.e., the Heaviside step function) and collects the following eigenvalue problem
\begin{equation}\label{eigenphi}
\left\{
\begin{array}{ll}
-\varphi''=\frac{\lambda}{x(1-x)}\varphi,&x\in(0,1), \\
\varphi(0)=\varphi(1)=0,&\\
\end{array}
\right.
\end{equation}
the principle eigenvalue of which determines the exponential convergence rate of \eqref{driftdiff} to the unique steady state (e.g., Theorem 3 in \cite{Chalub2009}).  One notes that $\lambda$ in \eqref{eigenphi} has a Rayleigh's variational quotient as
\[\lambda=\inf_{\varphi(0)=\varphi(1)=0,\varphi\not\equiv0}\frac{\int_0^1 (\varphi')^2dx}{\int_0^1 x(1-x) \varphi^2 dx},\]
and Theorem 262 in \cite{HLP1934} implies that $\lambda\geq2$, with the principal eigenvalue $\lambda=2$ achieved at $\varphi(x)=Cx(1-x)$ for some $C\in\mathbb R$.  This indicates that the exponential convergence rate in time of our discrete scheme in \cref{theorem_decay1} is optimal since the decay rate cannot be better for generic initial data.  Our numerical simulations in the coming section support this conclusion.
\end{remark}
The second main result of this paper goes as follows.
\begin{theorem}\label{theorem_decay2}
The solution of the numerical scheme \eqref{Euler} for the Replicator-Diffusion equation \eqref{driftdiff} with a monotone fitness potential ($V'$ either non-positive or non-negative)  converges exponentially to the steady state of \eqref{driftdiff} with exponential rate $\kappa$ (the diffusion rate) as the time step size $\tau$ goes to zero modulo $O(h)$ terms. More precisely,
the numerical scheme \eqref{Euler} for the general replicator-dynamics equation \eqref{driftdiff} with a monotone fitness potential satisfies
$$
\|\Phi^k-\Phi^*\|^2 \leq \|\Phi^0-\Phi^*\|^2 \left(\frac{1}{1+\kappa\tau}\right)^k+ O(h) \simeq C_0 \exp{(-\kappa k \tau)}+ O(h) \,,
$$
for $h$ small enough and $C_0$ both depending only on the distance of the initial data to the steady state $\|\Phi^0-\Phi^*\|$ and the bound of $V'$.
\end{theorem}
\begin{proof}
First of all, \cref{lemma_step2} implies that the sequence given by \eqref{phistar}-\eqref{am} is an equilibrium of \eqref{Euler}.  With $|V'(\cdot)|\leq M_1$ in $[0,1]$ for some constant $M_1>0$, the same calculations in \eqref{L2} lead to
\begin{align*}
A_n =&\,-\frac{h\kappa}{2}\sum_{i=1}^{N-1} (\Phi_i)^2(1-\Phi_i)\left( \frac{1}{\Phi_{i+1}-\Phi_i}- \frac{1}{\Phi_{i}-\Phi_{i-1}} \right)
   -\frac{h\kappa}{2}\sum_{i=1}^{N-1}\Phi_i(1-2\Phi_i)-h\sum_{i=1}^{N-1}(\Phi_i)^2(1-\Phi_i)V'(\Phi_i) \\
   \leq&\, \frac{h\kappa}{2}\sum_{i=1}^{N-1}  \Phi_{i}  -\frac{h\kappa}{2}\sum_{i=1}^{N-1}  (\Phi_{i})^2+M_1h\sum_{i=1}^{N-1}(\Phi_i)^2(1-\Phi_i),
\end{align*}

\begin{align*}
  B_n =&\,\frac{ah\kappa}{2}\Phi_{m+1}(1-\Phi_{m+1})\left(\frac{1}{\Phi_{m+2}-\Phi_{m+1}}- \frac{1}{\Phi_{m+1}-\Phi_{m}} \right)
  +\frac{ah\kappa}{2}(1-2\Phi_{m+1})+ah\Phi_{m+1}(1-\Phi_{m+1})V'(\Phi_{m+1}) \\
   \leq&\, \frac{ah\kappa}{2} \frac{\Phi_{m+1}(1-\Phi_{m+1})}{\Phi_{m+2}-\Phi_{m+1}}+\frac{ah\kappa}{2}(1-2\Phi_{m+1})+\frac{M_1ah}{4},
\end{align*}
and
 \begin{align*}
   C_n =&\,\frac{h\kappa}{2} \sum_{i=m+2}^{N-1} \Phi_i(1-\Phi_i)\left(\frac{1}{\Phi_{i+1}-\Phi_i}- \frac{1}{\Phi_{i}-\Phi_{i-1}} \right)
   +\frac{h\kappa}{2}\sum_{i=m+2}^{N-1}(1-2\Phi_i)+h \sum_{i=m+2}^{N-1}\Phi_i(1-\Phi_i)V'(\Phi_i) \\
    \leq&\,-\frac{ah\kappa}{2}\frac{\Phi_{m+1}(1-\Phi_{m+2})}{\Phi_{m+2}-\Phi_{m+1}}+M_1h \sum_{i=m+2}^{N-1}\Phi_i(1-\Phi_i).
 \end{align*}
Therefore, one finds
\begin{align}
   &\frac{1}{2\tau} \left(\|\Phi^{k+1}-\Phi^* \|^2 - \right. \left.\|\Phi^{k}-\Phi^* \|^2 \right) \notag\\
  \leq&\, \frac{h\kappa}{2}\sum_{i=1}^{N-1}  (\Phi_{i}  - \Phi_{i}^2)+\frac{ah\kappa}{2}(1-\Phi_{m+1})+M_1h\sum_{i=1}^{N-1}\Phi_i^2(1-\Phi_i)+\frac{M_1ah}{4}+M_1h \sum_{i=m+2}^{N-1}\Phi_i(1-\Phi_i)\notag\\
  =&-\frac{\kappa h}{2} \left[ \sum_{i=1}^{m} \Phi_i^2 + (\Phi_{m+1}-a)^2 +\sum_{i=m+2}^{N-1} ( \Phi_i-1)^2 \right]+\frac{h\kappa}{2}\sum_{i=1}^{m+1}  \Phi_{i}-\frac{h\kappa}{2}\sum_{i=m+2}^{N-1}  \Phi_{i}-\frac{3ah\kappa}{2} \Phi_{m+1}\notag\\
  &+\frac{(N+a+a^2-m-2)h\kappa}{2}+M_1h\sum_{i=1}^{N-1}(\Phi_i)^2(1-\Phi_i)+\frac{M_1ah}{4}+M_1h \sum_{i=m+2}^{N-1}\Phi_i(1-\Phi_i)\notag
\end{align}
\begin{align}  
  \leq & -\frac{\kappa}{2}\|\Phi^{k+1}-\Phi^*\|^2 +h\kappa \sum_{i=1}^{m+1}  \Phi_{i}+\frac{h\kappa}{2} \left((N-m)-\sum_{i=1}^{N-1}  \Phi_{i}  \right)\label{l2phi}\\
  &+M_1h\sum_{i=1}^{N-1}\Phi_i (1-\Phi_i)+\frac{M_1h}{4}+M_1h \sum_{i=m+2}^{N-1}\Phi_i(1-\Phi_i). \notag
\end{align}
In light of the concavity of $x(1-x)$, we can obtain from Taylor's expansion that
$$\Phi^{k+1}_i (1-\Phi^{k+1}_i)\leq \Phi_i^* (1-\Phi^*_i) +(\Phi^{k+1}_i -\Phi^*_i)  (1-2\Phi^*_i).$$
Substituting the inequality into \eqref{l2phi} yields
\begin{align}\label{l2decay}
     \frac{1}{2\tau} \left(\|\Phi^{k+1}-\Phi^* \|^2 - \|\Phi^{k}-\Phi^* \|^2 \right)
   \leq  &\,-\frac{\kappa}{2}\|\Phi^{k+1}-\Phi^*\|^2 +h(\kappa+3 M_1) \sum_{i=1}^{m+1}  \Phi_{i}\nonumber\\ &+h(\frac{\kappa}{2}+2M_1) \left((N-m)-\sum_{i=1}^{N-1}  \Phi_{i}  \right)+\frac{M_1h}{4}. 
\end{align}
Now, let us denote $d_{k+1}:=\sum_{i=1}^{N-1}\Phi^{k+1}_i$, then summing equation \eqref{Euler} with respect to $i$ yields
\begin{align}\label{dk}
  d_{k+1}-d_k &=-\tau\sum_{i=1}^{N-1}\Phi_i^{k+1}(1-\Phi_i^{k+1})V'(\Phi_i^{k+1}) \nonumber.
\end{align}

We next show the convergences of $\sum_{i=1}^{N-1}\Phi^{k+1}_i$ and $\sum_{i=1}^{m+1}\Phi^{k+1}_i$ provided that $V'(\cdot)$ is of one sign.  Case i). If $V'(\cdot)>0$.  Let us denote for simplicity that $c_{k+1}^1:=\sum_{i=1}^{m+1}\Phi^{k+1}_i$, then summing equation \eqref{Euler} with respect to $i$ yields
    \begin{align*}
  c_{k+1}^1-c_k^1 &=-\frac{\kappa \tau}{2} \frac{\Phi_{m+1}^{k+1}(1-\Phi_{m+2}^{k+1})}{\Phi^{k+1}_{m+2}-\Phi_{m+1}^{k+1}}-\tau\sum_{i=1}^{m+1}\Phi_i^{k+1}(1-\Phi_i^{k+1})V'(\Phi_i^{k+1})<0. 
\end{align*}
The fact that $0\leq c_{k}^1 \leq m+1$ for each $k\in \mathbb{N}$ implies that $\{c_k^1\}$ is Cauchy with the limit $a$.  Similarly, since $0\leq d_{k} \leq N-1$ and $d_{k+1}-d_k<0$, one can show that $\sum_{i=1}^{N-1}\Phi^{k+1}_i$ is Cauchy and its limit is $(N+a-m-2)$.
Case ii).  If $V'(\cdot)<0$.  Then we denote $c_{k+1}^2:=\sum_{i=m+2}^{N-1}\Phi^{k+1}_i=d_{k+1}-c_{k+1}^1$, then summing equation \eqref{Euler} with respect to $i$ yields
\begin{align*}
  c_{k+1}^2-c_k^2 &=\frac{\kappa \tau}{2} \frac{\Phi_{m+1}^{k+1}(1-\Phi_{m+2}^{k+1})}{\Phi^{k+1}_{m+2}-\Phi_{m+1}^{k+1}}-\tau\sum_{i=m+2}^{N-1}\Phi_i^{k+1}(1-\Phi_i^{k+1})V'(\Phi_i^{k+1})>0. 
\end{align*}
The fact that $0\leq c_{k}^2 \leq N-m-2$ for each $k\in \mathbb{N}$ implies that $\{c_k^2\}$ is Cauchy with the limit $(N-m-2)$.  Similarly, since $0\leq d_{k} \leq N-1$ and $d_{k+1}-d_k>0$, one can show that $\sum_{i=1}^{N-1}\Phi^{k+1}_i$ is Cauchy and its limit is $(N+a-m-2)$, then $\{c_k^1\}$ is Cauchy with the limit $a$.

Since $c_k^1\to a$ and $d_k\to (N+a-m-2)$ with $0\leq a\leq 1$, there exists some $k_1 \in\mathbb{N}$ such that for any $k>k_1$
\[\sum_{i=1}^{m+1}\Phi^{k+1}_i \leq \frac{5}{4},\quad  \sum_{i=1}^{N-1}\Phi^{k+1}_i \geq N-m-\frac{5}{2}.\]
Therefore from \eqref{l2decay} we deduce that $\frac{1}{2\tau} \left(\|\Phi^{k+1}-\Phi^* \|^2 - \|\Phi^{k}-\Phi^* \|^2 \right)\leq -\frac{\kappa}{2}\|\Phi^{k+1}-\Phi^* \|^2+(\frac{5 \kappa}{2} +9 M_1)h$, and then
\begin{align*}
  \|\Phi^k-\Phi^*\|^2 \leq& \left(\frac{1}{1+\kappa \tau}\right)^k \|\Phi^0-\Phi^*\|^2  +(5 +\frac{18 M_1}{\kappa})  \left[ 1-\left(\frac{1}{1+\kappa\tau}\right)^k \right] h\\
   =& \left(\|\Phi^0-\Phi^*\|^2-(5 +\frac{18 M_1}{\kappa}) h \right) \left(\frac{1}{1+\kappa \tau}\right)^k+(5 +\frac{18 M_1}{\kappa}) h.
\end{align*}
Choose $h\leq \|\Phi^0-\Phi^*\|^2/(5 +\frac{18 M_1}{\kappa}) $, then for $\tau\approx0$ we have $\left(\frac{1}{1+\kappa \tau}\right)^k \approx \exp (-\kappa k\tau)\approx \exp(-\kappa t)$.
\end{proof}

\section{Numerical Studies}
We now present several sets of numerical simulations to illustrate and verify the analytical properties of the proposed scheme and to demonstrate further spatio-temporal dynamics within the generic drift which are not captured by the analysis.  Note that the time-dependent solution converges to the Dirac-delta function in the long-time limit. In our map variables, this corresponds to the convergence to a stationary map $\Phi_\infty(\eta)$ given by a step function. To deal with the issue that discrete values near the boundary are too close to be distinguished from each other under machine precision, we set the following additional boundary criteria throughout our simulations
\begin{equation}\label{bdc}
\Phi_i^k=\left\{
\begin{array}{ll}
0,& \text{for} \; 0<\Phi_i^k<10^{-10}, \\
1,&\text{for} \; 1-10^{-10}<\Phi_i^k<1.\\
\end{array}
\right.
\end{equation}

\subsection{Diffusion Equation}
We first numerically study the purely diffusive model \eqref{diff}.  The initial data are chosen to be either the symmetric $f_0(x)=\frac{\pi }{\pi-2}(1-\sin(\pi x))$ or asymmetric $f_0(x)=2 x$ so we test the robustness of the proposed scheme with respect to initial data regarding the long-time dynamics. We remind the reader that the conserved quantity \eqref{conser} determines asymptotically the proportion of the mass allocated to each end of the interval, that is the jump location in mass variables, see Theorem \ref{theorem21}.  To study the convergence of the numerical scheme, we consider the discrete $L^1$ error for $\Phi_{i,h}^{k,\tau}$, $i=0,1,...,N$ as follows 
\begin{equation*}
    E_h^{\tau}(n\tau)=h\sum_{i=0}^N\left| \Phi_{i,h}^{k,\tau} -\Phi_{i}^{k,\text{ref}}\right|,
\end{equation*}
where the reference solution $\Phi_{i}^{k,\text{ref}}$ is obtained with mesh size $h=1/1600$ and time increment $\tau=0.0005$. Notice that this quantity is equivalent to an approximation of the Monge-Kantorovich or 1-Wasserstein transport distance between the associated densities. Moreover, the experimental order of convergence (EOC) in spatial-discretization $h$ is defined as 
\begin{equation*}
\text{EOC}_h:=\log_2(E_h)-\log_2(E_{h/2}),
\end{equation*}
and the experimental order of  convergence (EOC) in time-discretization $\tau$ is defined as 
\begin{equation*}
\text{EOC}_{\tau}:=\log_2(E^{\tau})-\log_2(E^{\tau/2}).
\end{equation*}
Table \ref{tablemesh} presents the errors and EOCs in spatial-discretization computed at a final time $T=0.992$, a multiple of $\tau=0.016$, with constant time increment $\tau=0.0005$.  We can see that difference errors decrease as the mesh is refined and a finer spatial-discretization results in higher accuracy. Besides, the orders of convergence in spatial-discretization indicate a convergence order of greater than one in both cases, and EOCs increase for a finer mesh.  The results from both symmetric and asymmetric initial data indicate that the proposed scheme is robust to initial data.  Table \ref{tabletime} shows the errors and EOCs in time-discretization with a fixed space mesh $h=1/1600$.  Similar robust convergence witnesses that the method converges as the time mesh shrinks. 
 
\renewcommand\arraystretch{1.1}
\begin{table}[H]
\centering
\caption{Robust convergence in spatial-discretization $h$ with $\tau=0.0005$ up to a terminal time $T=0.992$.}
\setlength{\tabcolsep}{5mm}
\begin{tabular}{c|cccccccccc}
\hline
                              & \multicolumn{2}{c}{$f_0(x)=\frac{\pi }{\pi-2}(1-\sin(\pi x))$}                 & \multicolumn{2}{c}{$f_0(x)=2x$} \\ \hline
\multicolumn{1}{c|}{$h$}      & Error  $E_h^{\tau}$   & \multicolumn{1}{c|}{$\text{EOC}_h$}   & Error  $E_h^{\tau}$       & $\text{EOC}_h$    \\ \hline
1/50                   & 1.021e-02                 & \multicolumn{1}{c|}{}      & 1.061e-02                 &                         \\
1/100                  & 5.337e-03                 & \multicolumn{1}{c|}{0.936} & 5.344e-03                 & 0.989                   \\
1/200                  & 2.597e-03                 & \multicolumn{1}{c|}{1.039} & 2.597e-03                 & 1.041                   \\
1/400                  & 1.157e-03                 & \multicolumn{1}{c|}{1.167} & 1.157e-03                 & 1.166                   \\
1/800                  & 4.010e-04                 & \multicolumn{1}{c|}{1.529} & 4.026e-04                 & 1.523                   \\ \hline
\end{tabular}
\label{tablemesh}
\end{table}

\renewcommand\arraystretch{1.1}
\begin{table}[H]
\centering
\caption{Robust convergence in time-discretization $\tau$ with $h=1/1600$ up to a terminal time $T=0.992$}.
\setlength{\tabcolsep}{5mm}
\begin{tabular}{c|cccccccccc}
\hline
                              & \multicolumn{2}{c}{$f_0(x)=\frac{\pi }{\pi-2}(1-\sin(\pi x))$}                 & \multicolumn{2}{c}{$f_0(x)=2x$} \\ \hline
\multicolumn{1}{c|}{${\tau}$}      & Error  $E_h^{\tau}$   & \multicolumn{1}{c|}{$\text{EOC}_{\tau}$}   & Error  $E_h^{\tau}$       & $\text{EOC}_{\tau}$    \\ \hline
0.016                  & 4.191e-02                 & \multicolumn{1}{c|}{}      & 4.443e-02                 &                         \\
0.008                  & 3.058e-02                 & \multicolumn{1}{c|}{0.455} & 3.221e-02                 & 0.464                   \\
0.004                  & 2.000e-02                 & \multicolumn{1}{c|}{0.612} & 2.139e-02                 & 0.590                   \\
0.002                  & 1.042e-02                 & \multicolumn{1}{c|}{0.940} & 1.160e-02                 & 0.884                   \\
0.001                  & 3.604e-03                 & \multicolumn{1}{c|}{1.532} & 4.295e-03                 & 1.433                   \\ \hline
\end{tabular}
\label{tabletime}
\end{table}

\begin{figure}[htb]
\centering
\captionsetup{width=.85\linewidth}
\subfloat[$f_0(x)=\frac{\pi }{\pi-2}(1-\sin(\pi x))$]{\includegraphics[width=0.45\columnwidth]{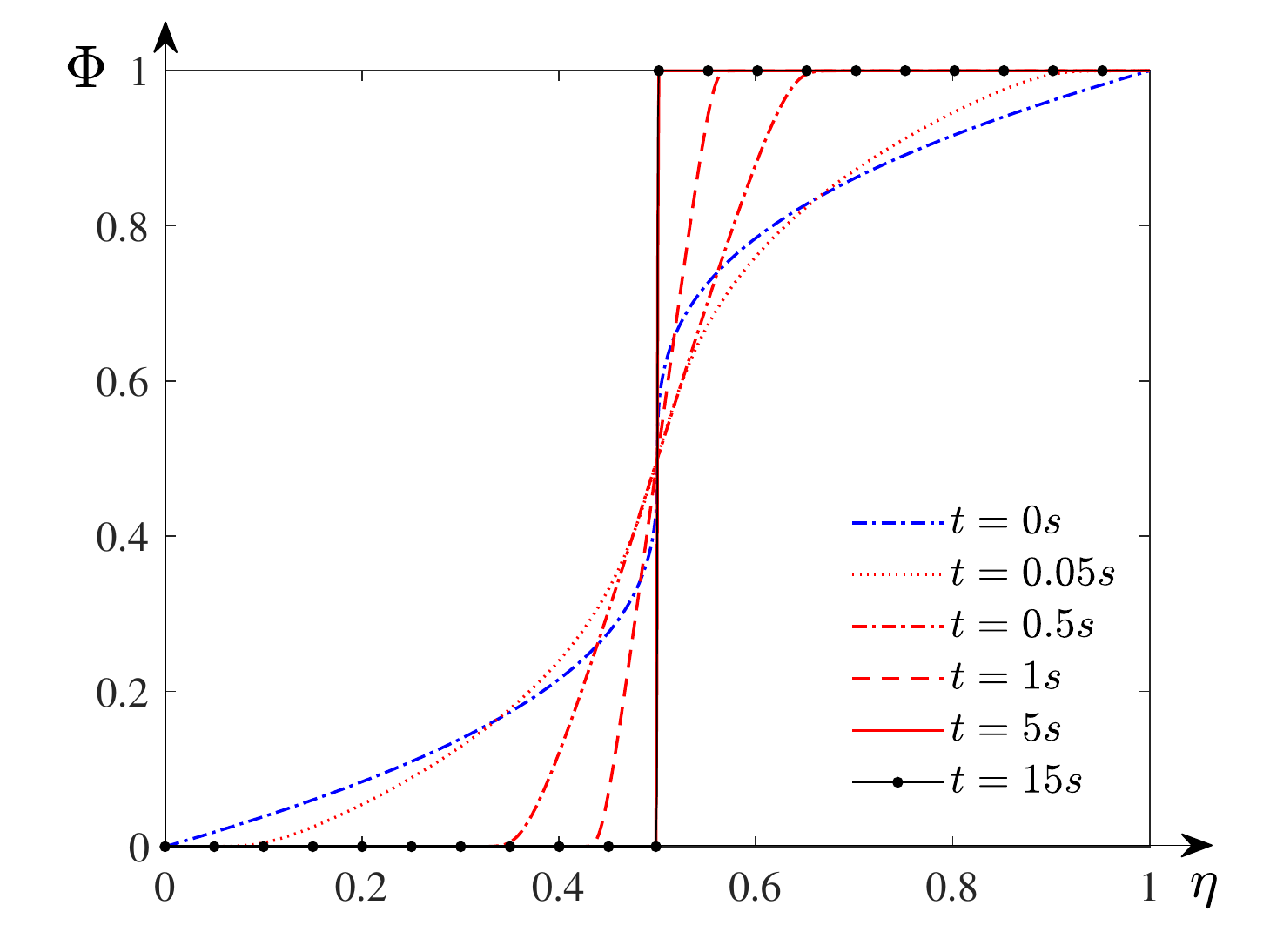}}\hspace{-5mm}
\subfloat[$f_0(x)=2 x$]{\includegraphics[width=0.45\columnwidth]{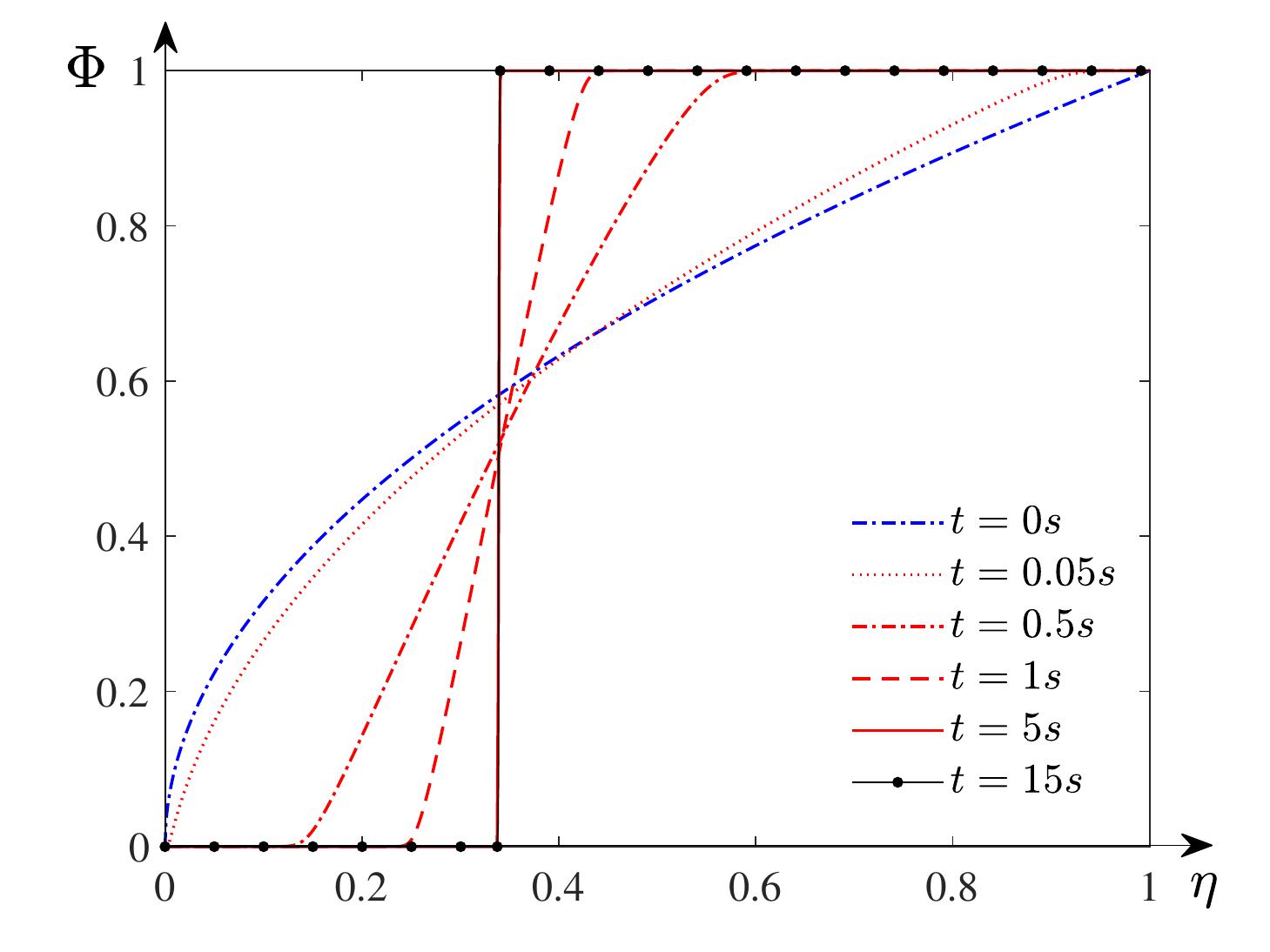}}
\caption{Evolution of the inverse cumulative distribution function $\Phi$ for the Diffusion-equation \eqref{diff} subject to symmetric initial data on the left and asymmetric initial data on the right.  Both experiments capture the dominance of dynamics by the Dirac-delta as stated in Theorem \ref{theorem21}. In particular, one finds that the degeneracy of diffusion at the endpoints makes the spatio-temporal dynamics of (\ref{diff}) relatively simple in light of the proposed mass transport method now that the singularities can be quantified.}\label{fig:inverse}
\end{figure}

Next we present the evolution of \eqref{diff} out of the initial data $f_0(x)=\frac{\pi -2}{\pi}(1-\sin(\pi x))$ and $f_0(x)=2 x$ with the step size $h=1/999$ and time step $\tau=0.001$. Figure \ref{fig:inverse} captures the convergence to the steady-state of the numerical scheme given by a Heaviside step function, with at most one intermediate value, as stated in Lemma \ref{lemma_step1} or Lemma \ref{theorem21}.  We note that the jump locations $\eta_0=\frac{1}{2}$ and $\frac{1}{3}$ for these two initial data, respectively, and this, together with the conservation of center of mass, are well preserved by our method. 

\begin{figure}[htb]
\centering
\captionsetup{width=.85\linewidth}
  \begin{subfigure}{.45\textwidth}
    \includegraphics[width = \linewidth]{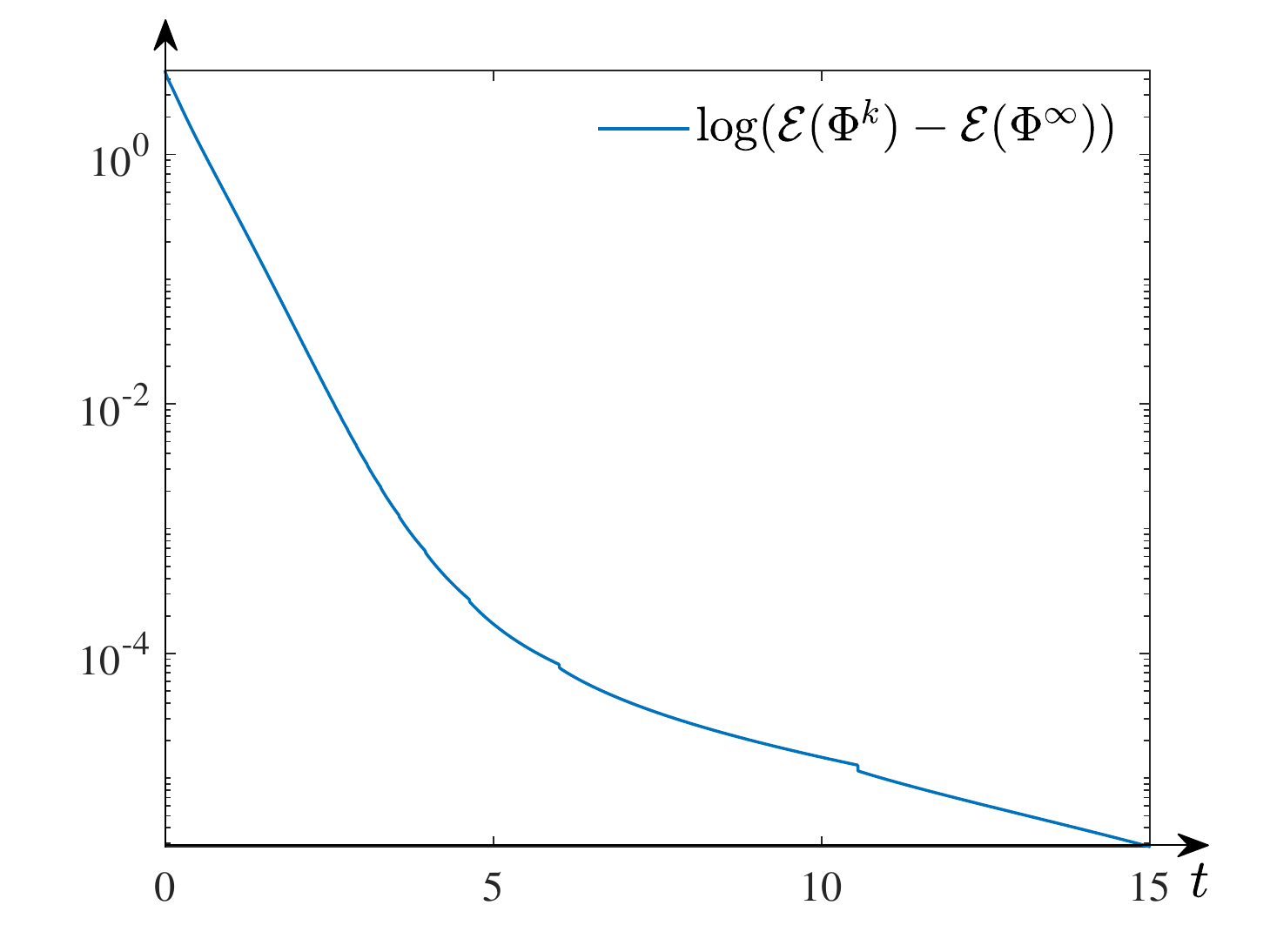}
  \end{subfigure}%
  \begin{subfigure}{.45\textwidth}
    \centering
    \includegraphics[width = \linewidth]{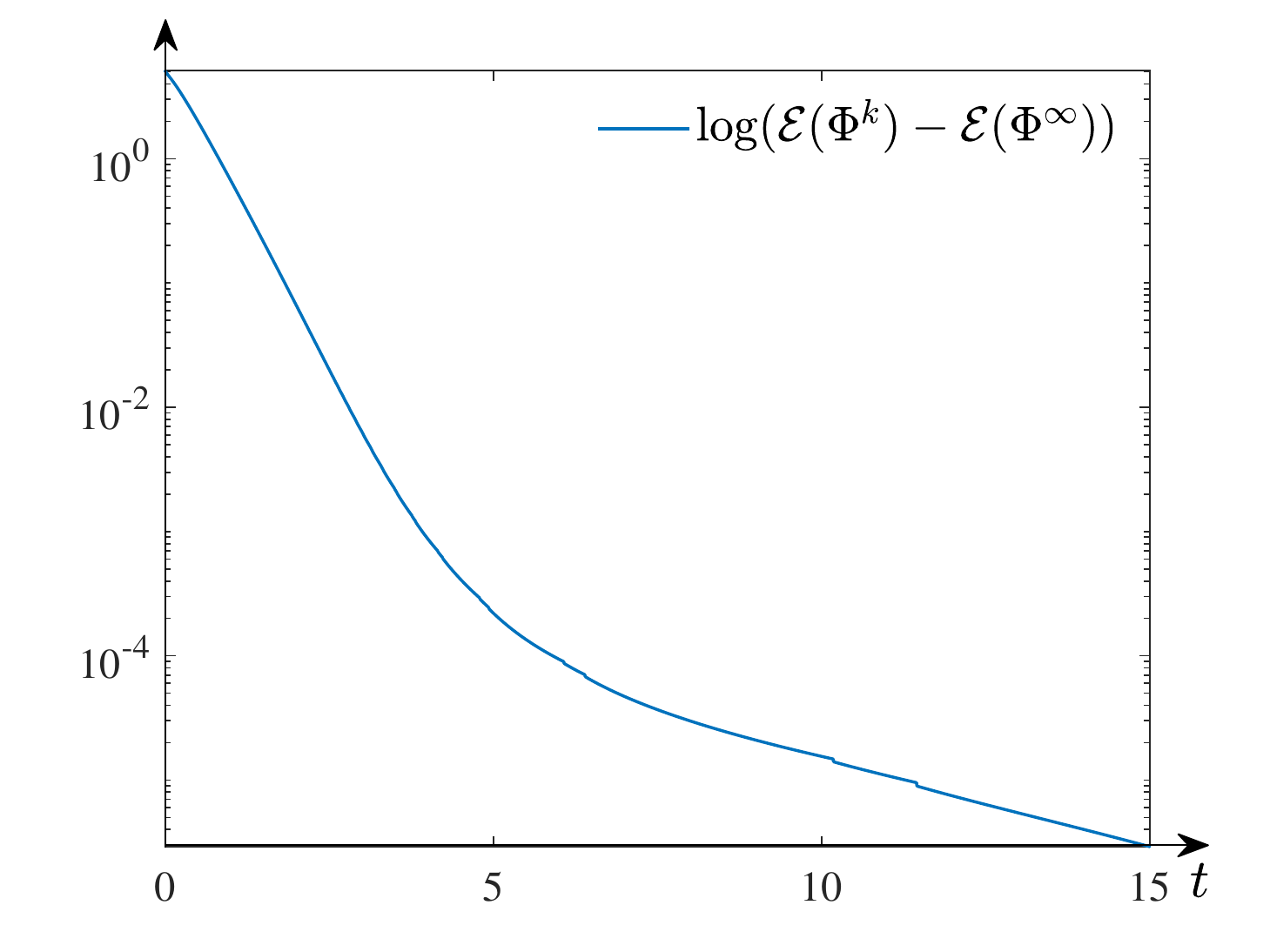}
  \end{subfigure}
  \\
  \begin{subfigure}{.45\textwidth}
    \centering
    \includegraphics[width = \linewidth]{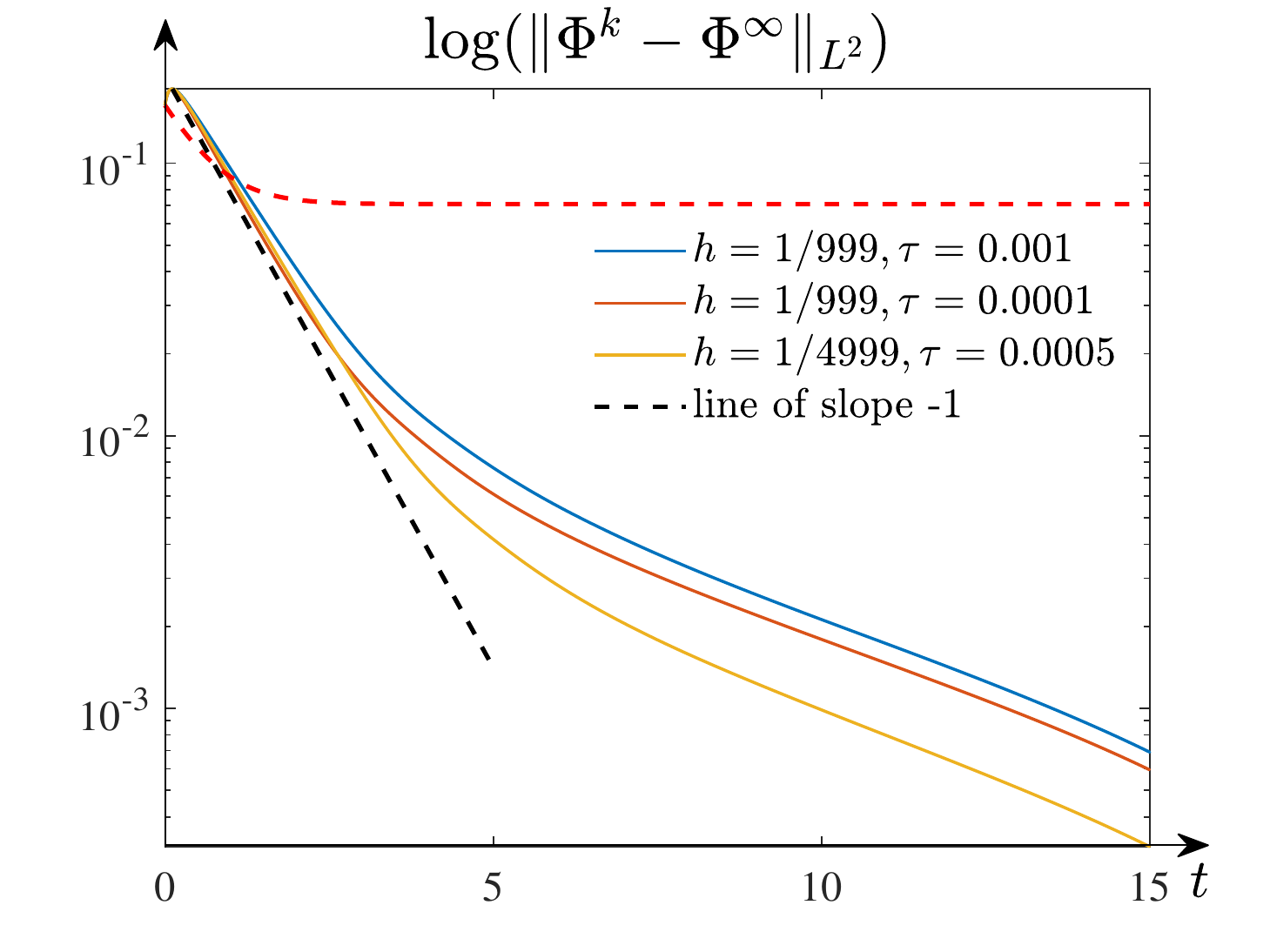}
    \caption{$f_0(x)=\frac{\pi }{\pi-2}(1-\sin(\pi x))$}
  \end{subfigure}%
  \begin{subfigure}{.45\textwidth}
    \centering
    \includegraphics[width = \linewidth]{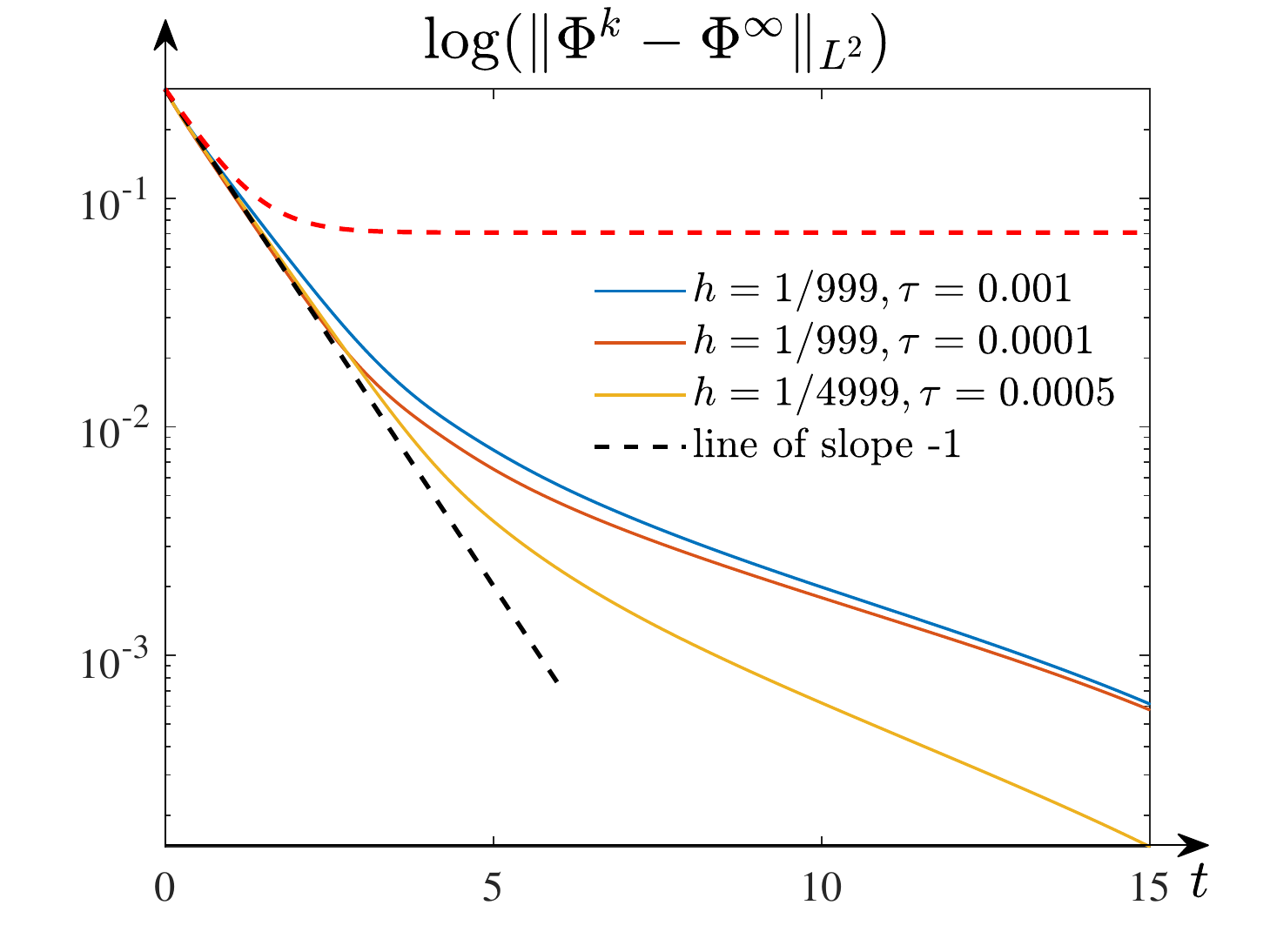}
    \caption{$f_0(x)=2x$}
  \end{subfigure}
  \caption{Convergence of the energy and error for diffusion equation \eqref{diff}.  \textbf{Top}: Dissipation of the discrete energy $\mathcal{E}(\Phi)$ in \eqref{discreteenergy} to that of the (proxy) steady state in logarithmic scale.  \textbf{Bottom}: $L^2$-distance between the inverse distribution function and the numerically computed steady state (as a proxy) in logarithmic scale.  One finds that the (exponential) decay rate of the discrete scheme approaches 1 as space or time mesh shrinks.  These findings agree well with Theorem \ref{theorem_decay1}.}
 \label{fig:energy}
\end{figure}

Figure \ref{fig:energy} presents two additional sets of biologically relevant and computationally favored properties of the scheme.  For instance, the exponential dissipation of energy \eqref{discreteenergy} in time in Figure \ref{fig:energy}-(\textbf{Top}) readily indicates that the Euler implicit scheme \eqref{Euler} carries that from the convex splitting scheme \eqref{full} and the continuum equation. The quantities in Figure \ref{fig:energy} showing the decay towards steady states are computed by choosing the solution of our scheme for a longer time as the numerically computed (proxy) stationary solution. This is essential since the energy of a Heaviside steady state diverges to infinity in both the continuous and discrete equation according to our discussion in Lemma \ref{lemma35} and the subsequent remarks.  The evolution of the $L^2$-norm of the difference between the numerical solution and the numerically computed steady state is plotted in Figure \ref{fig:energy}-(\textbf{Bottom}). Notice that this quantity is equivalent to the 2-Wassertein transport distance between the associated densities. We can find that the convergence is faster than the rate obtained in Theorem \ref{theorem_decay1} -- exponential convergence in time modulo $O(h)$ terms.  This confirms the derivation of Theorem \ref{theorem_decay1} and the existence of time modulo $O(h)$ terms. We can also see that the convergence rates depend on $\tau$ and $h$, and as either one shrinks, the convergence of the $L^2$-norm fits better $e^{-t}$ initially, which is consistent with the theoretical analysis, however, the decay rate is dramatically influenced by the $O(h)$-error in Theorem \ref{theorem_decay1} as time evolves.

\subsection{Replicator-Diffusion Equation}
We proceed to study the full Replicator-Diffusion equation \eqref{driftdiff} that includes the forward Kimura equation with frequency selection.  In particular, we restrict our attention to the choice that $V'(x)=\alpha x+\beta$ for some constant $\alpha$ and $\beta$.  Throughout the simulations we fix $f_0(x)=6x(1-x)$ with the space mesh size $h=1/999$ and time step size $\tau=10^{-3}$, and then study the variation of the diffusion rate and fitness potential, starting with $\kappa=2$, $V'(x)=4x+2$ and $\kappa=4$, $V'(x)=-4x+2$ as examples.   

One observes from Figure \ref{figv}-(a) that the numerical solution converges to the step function with the discrete conservation laws holding. As shown in Figure \ref{figv}-(b), there exist some jumps in the dissipation of the entropy, which can be attributed to the additional boundary criteria \eqref{bdc}. One can observe the smearing of these jumps by decreasing the tolerance in \eqref{bdc}.  However, this tolerance value cannot be chosen too small either for Newton's iterations to converge. Finally, Figure \ref{figv}-(c) shows the decay in the $L^2$-distance or 2-Wassertein transport distance towards equilibration. One again finds that the (exponential) decay rate of the discrete scheme approaches 1 as the space or time mesh shrinks for short times. This finding agrees well with Theorem \ref{theorem_decay2}. We also observe in Figure \ref{figv}-(\textbf{Top}) that the asymmetric initial data leads to an interesting phenomena of different equilibration time scales, a first transient slower time scale in which all the mass from the left is pushed towards 0 followed by a faster decay towards the equilibrium for large times. 

\begin{figure}[htb]
  \centering
\captionsetup{width=.85\linewidth}
  \begin{subfigure}{.327\textwidth}
    \centering
    \includegraphics[width =\linewidth]{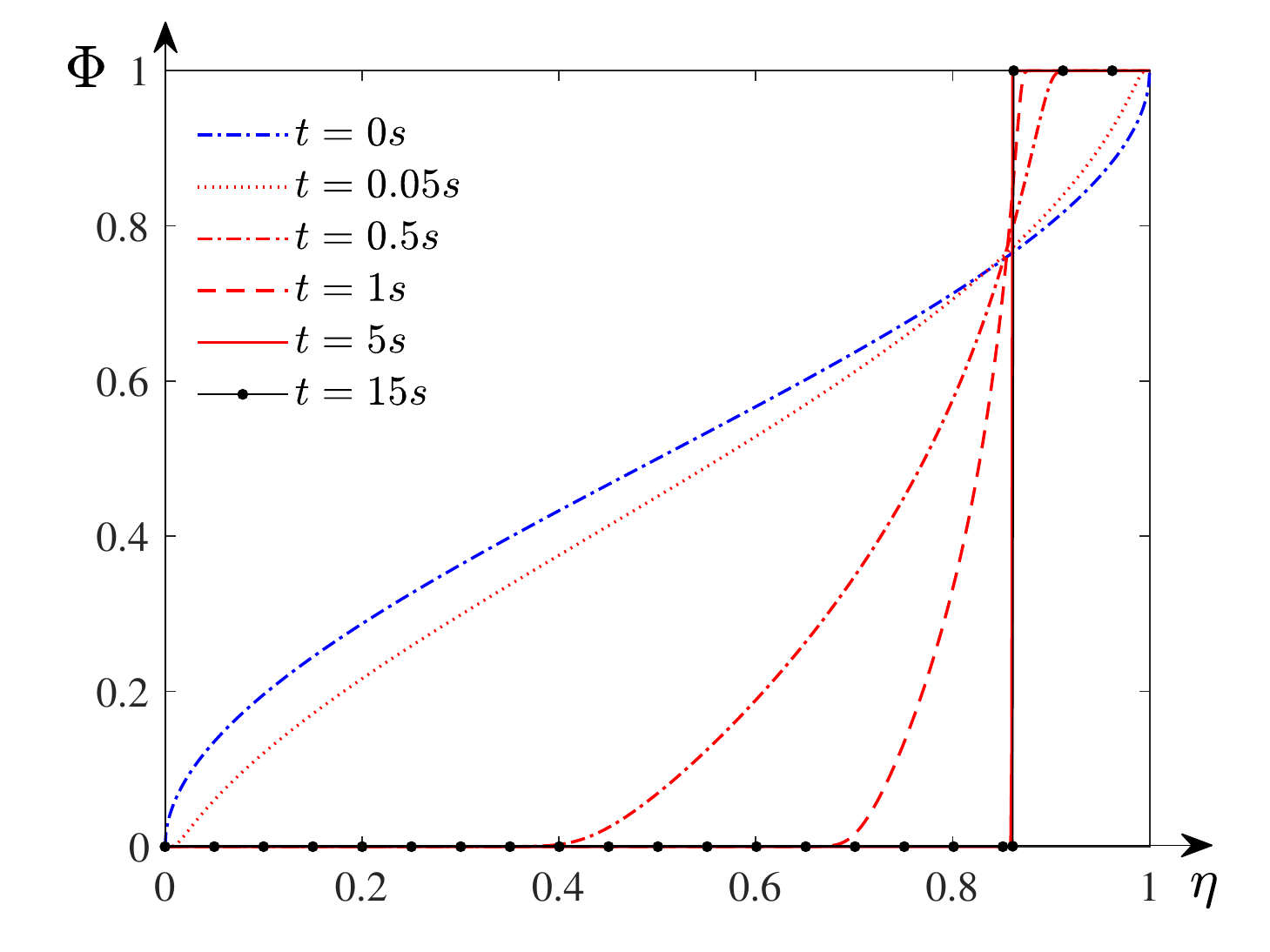}
  \end{subfigure}
  \hspace{0.01in}
  \begin{subfigure}{.327\textwidth}
    \centering
    \includegraphics[width = \linewidth]{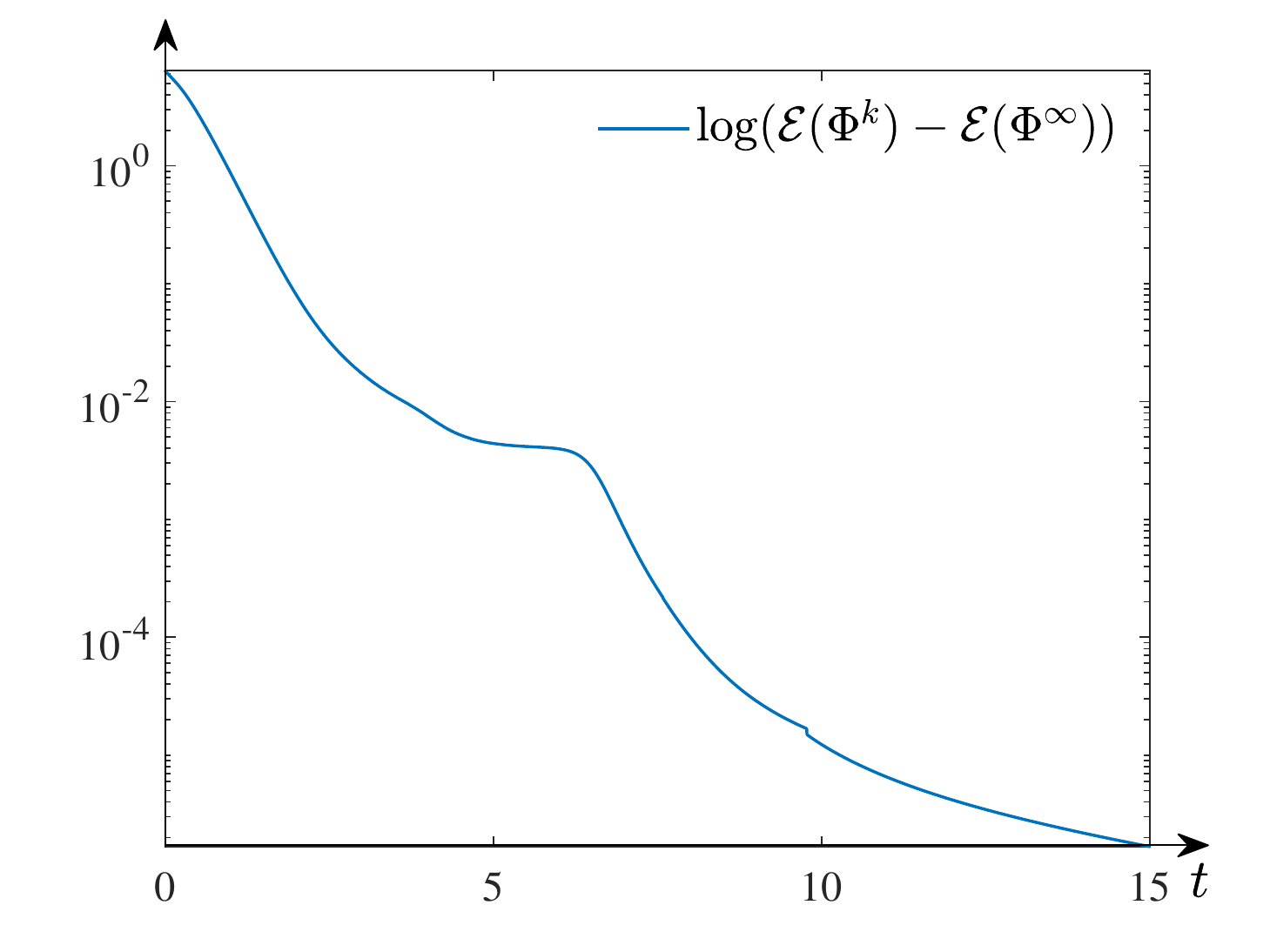}
  \end{subfigure}%
  \hspace{0.01in}
  \begin{subfigure}{.327\textwidth}
    \centering
    \includegraphics[width = \linewidth]{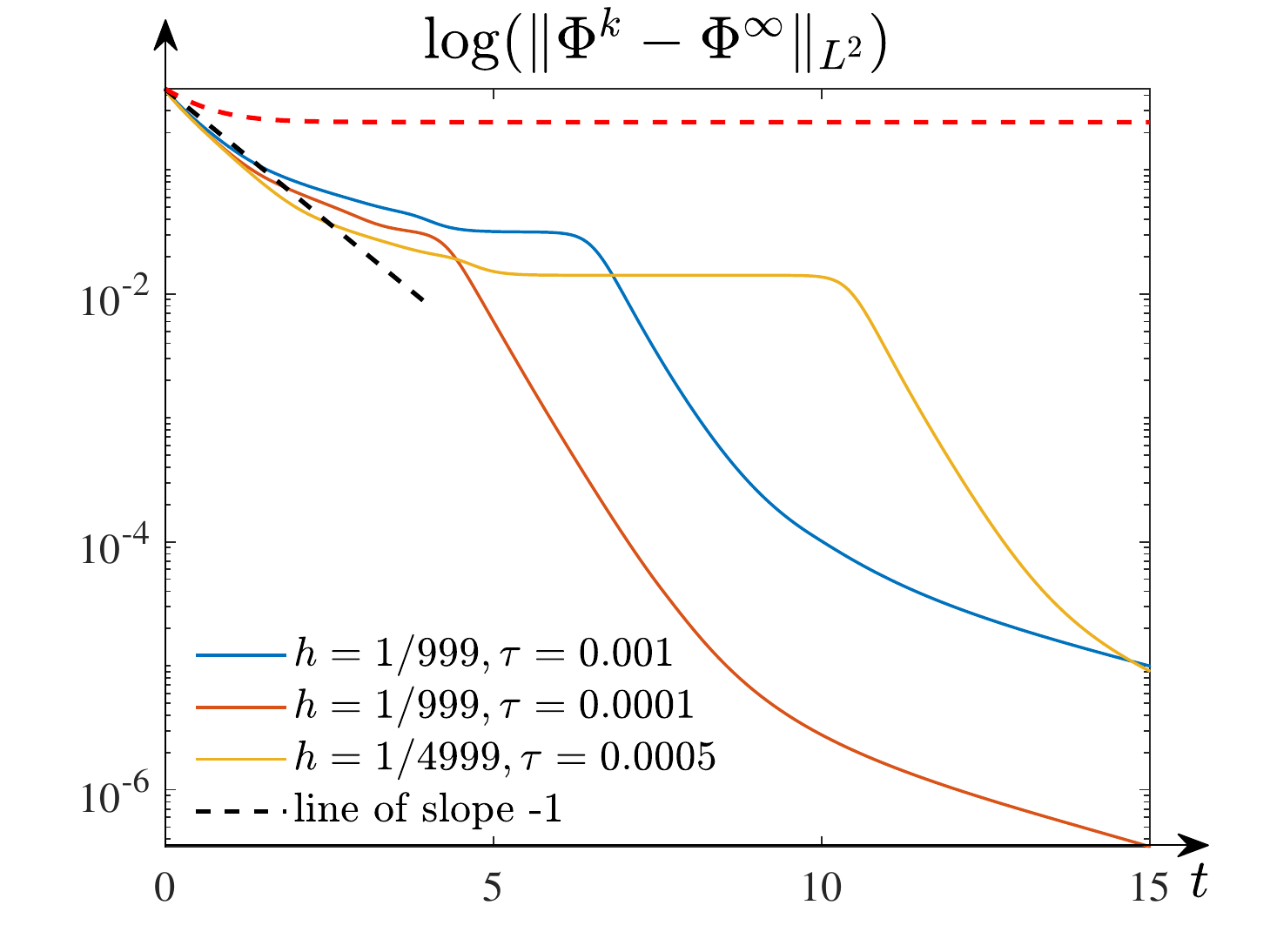}
  \end{subfigure}
  \begin{subfigure}{.327\textwidth}
    \centering
    \includegraphics[width =\linewidth]{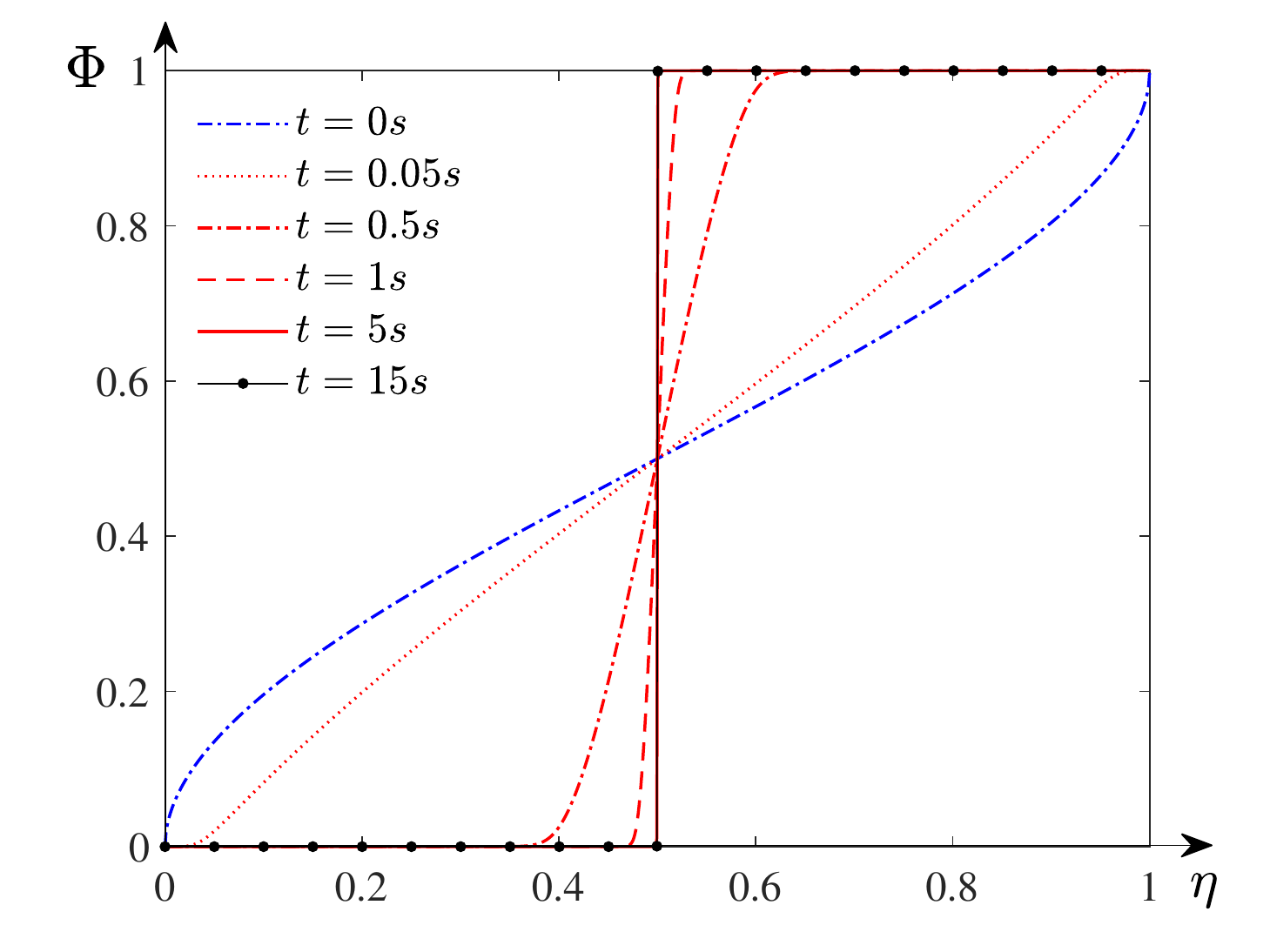}
    \caption{Convergence to steady state}
  \end{subfigure}%
  \hspace{0.01in}
  \begin{subfigure}{.327\textwidth}
    \centering
    \includegraphics[width = \linewidth]{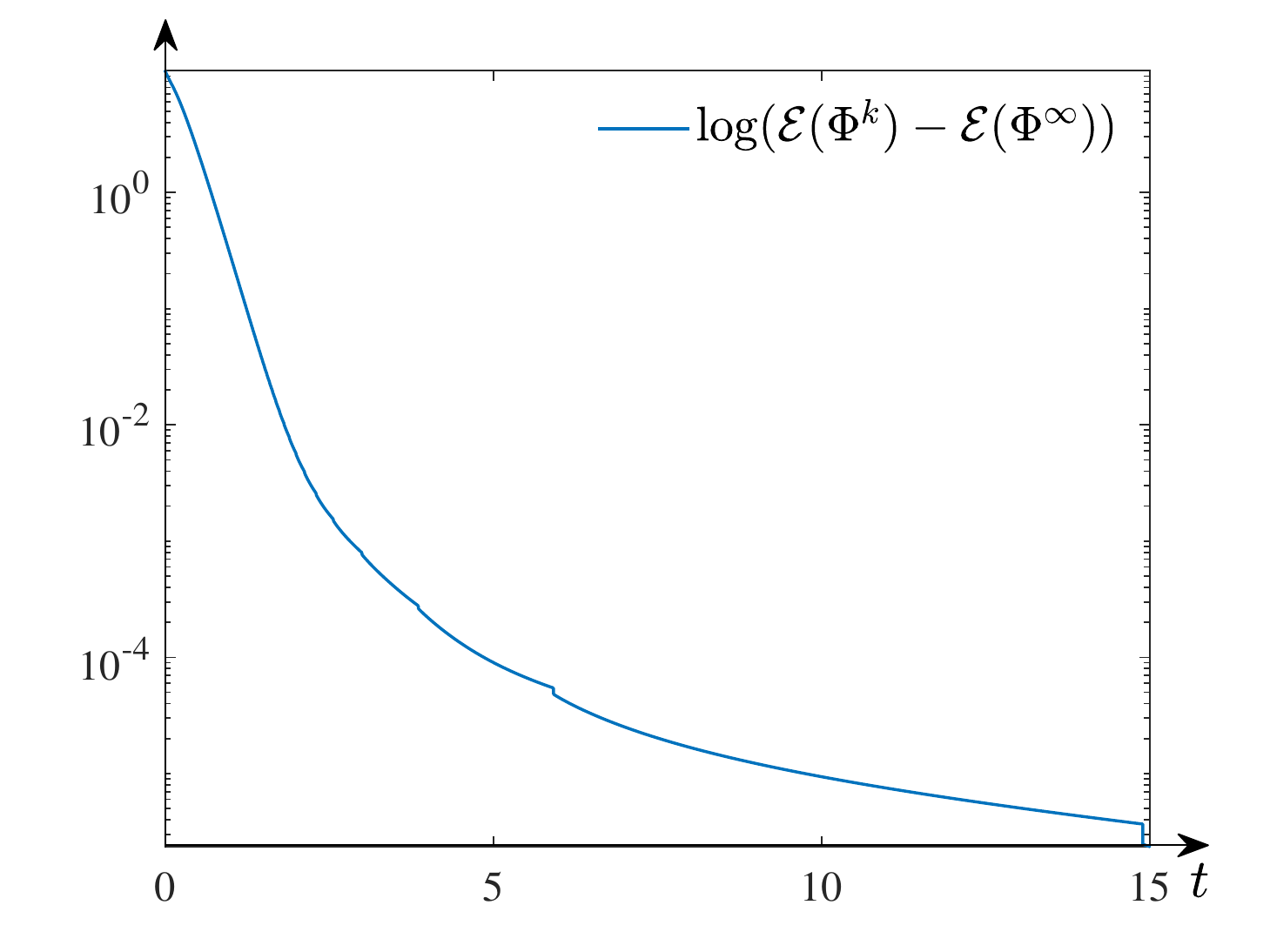}
    \caption{Stabilization of free energy}
  \end{subfigure}%
  \hspace{0.01in}
  \begin{subfigure}{.327\textwidth}
    \centering
    \includegraphics[width = \linewidth]{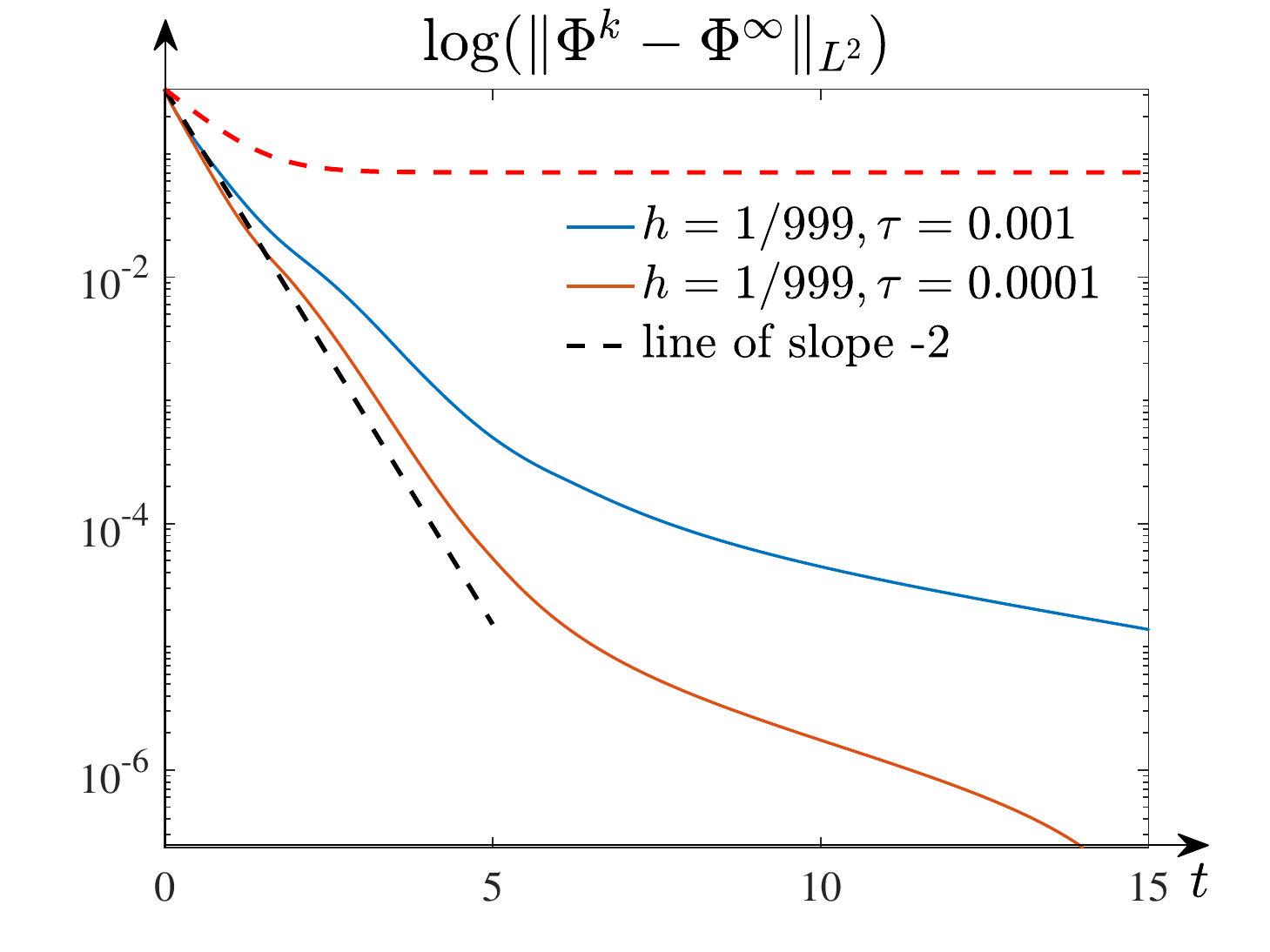}
    \caption{Exponential order of convergence}
  \end{subfigure}
  \caption{\textbf{Top}: $\kappa=2$ and $V'(x)=4x+2$. \textbf{Bottom}: $\kappa=4$ and $V'(x)=-4x+2$.  All simulations are conducted out of the initial data $f_0(x)=6x(1-x)$.  (a): Convergence of the inverse cumulative distribution function $\Phi$.
  (b): Dissipation of the discrete energy $\mathcal{E}(\Phi)$ in \eqref{discreteenergy} to that of the (proxy) steady state in logarithmic scale.
  (c): $L^2$-distance between the inverse distribution function and the numerically computed steady state (as a proxy) in logarithmic scale. }\label{figv}
\end{figure}
We remark that the proofs of our main results rely on the assumption that $V'(x)$ is of one sign for technique purpose, and the numerics in the bottom of Figure \ref{figv} favor the application of our scheme otherwise.  However, we have to exclude the relaxation of this assumption in this work, unfortunately. 

\subsection{Spatial-Temporal Dynamics}
We conclude this section by presenting several additional sets of numerical experiments on the full model \eqref{driftdiff} to illustrate and verify our main results, and to capture the spatio-temporal dynamics of the random genetic drift when analytical tools lack.  For these purposes, we choose featured initial data to showcase the dynamics of generic shrift to the variation of the potential function.  For simplicity, we assume $\kappa=2$ and consider two different fitness potential $V'(x)=2$ and $V'(x)=-3x+1$.  The step size $h=1/999$ and time step $\tau=0.001$ are fixed for all the numerical experiments. Each row in Figure \ref{fig:phi} shows the evolution of the inverse cumulative distribution function $\Phi$ of \eqref{driftdiff} corresponding to the initial data depicted in each column.  The dynamics are soon dominated by the Dirac-delta concentrations at the ends of the interval in each case, though one observes relatively different evolutionary transient phenomena such as merging of bumps and different equilibration speeds.  All these findings point to the computational/mathematical evidence that the evolutionary changes at the molecular level are caused by random genetic drift when described by a Moran process.

\begin{figure}[htb]
  \centering
\captionsetup{width=.85\linewidth}
  \begin{subfigure}{.327\textwidth}
    \centering
    \includegraphics[width =\linewidth]{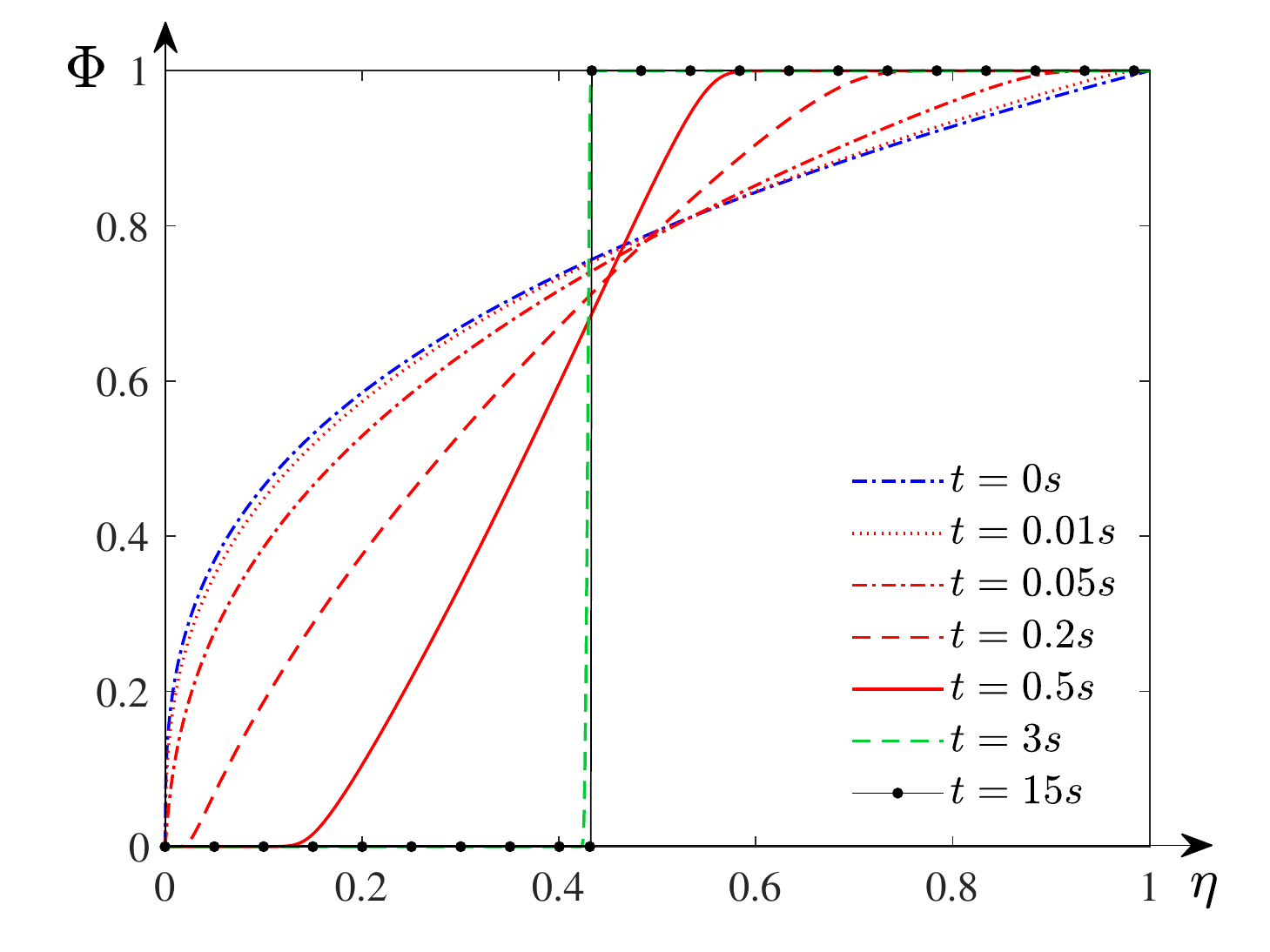}
  \end{subfigure}
  \hspace{0.01in}
  \begin{subfigure}{.327\textwidth}
    \centering
    \includegraphics[width = \linewidth]{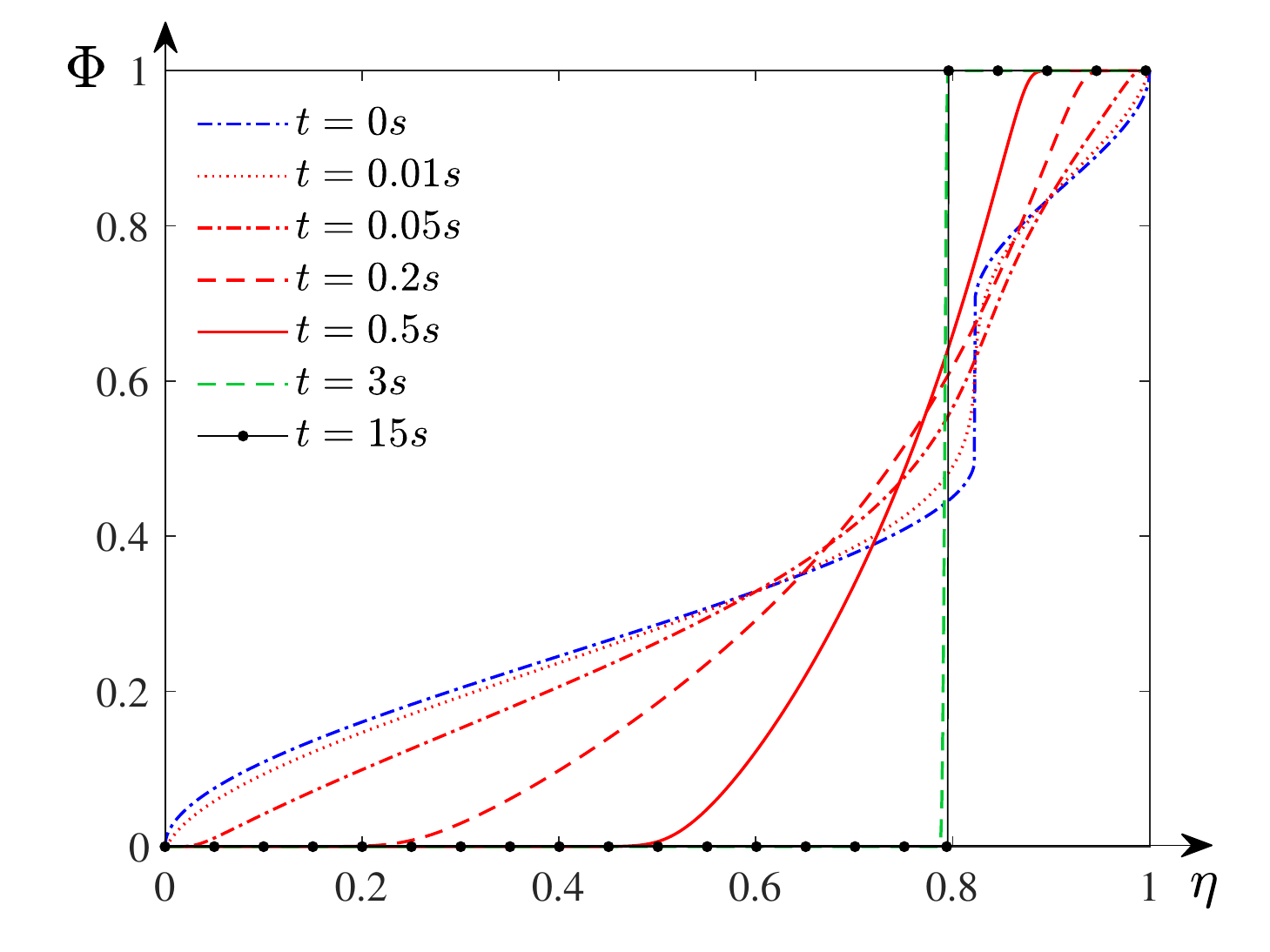}
  \end{subfigure}%
  \hspace{0.01in}
  \begin{subfigure}{.327\textwidth}
    \centering
    \includegraphics[width = \linewidth]{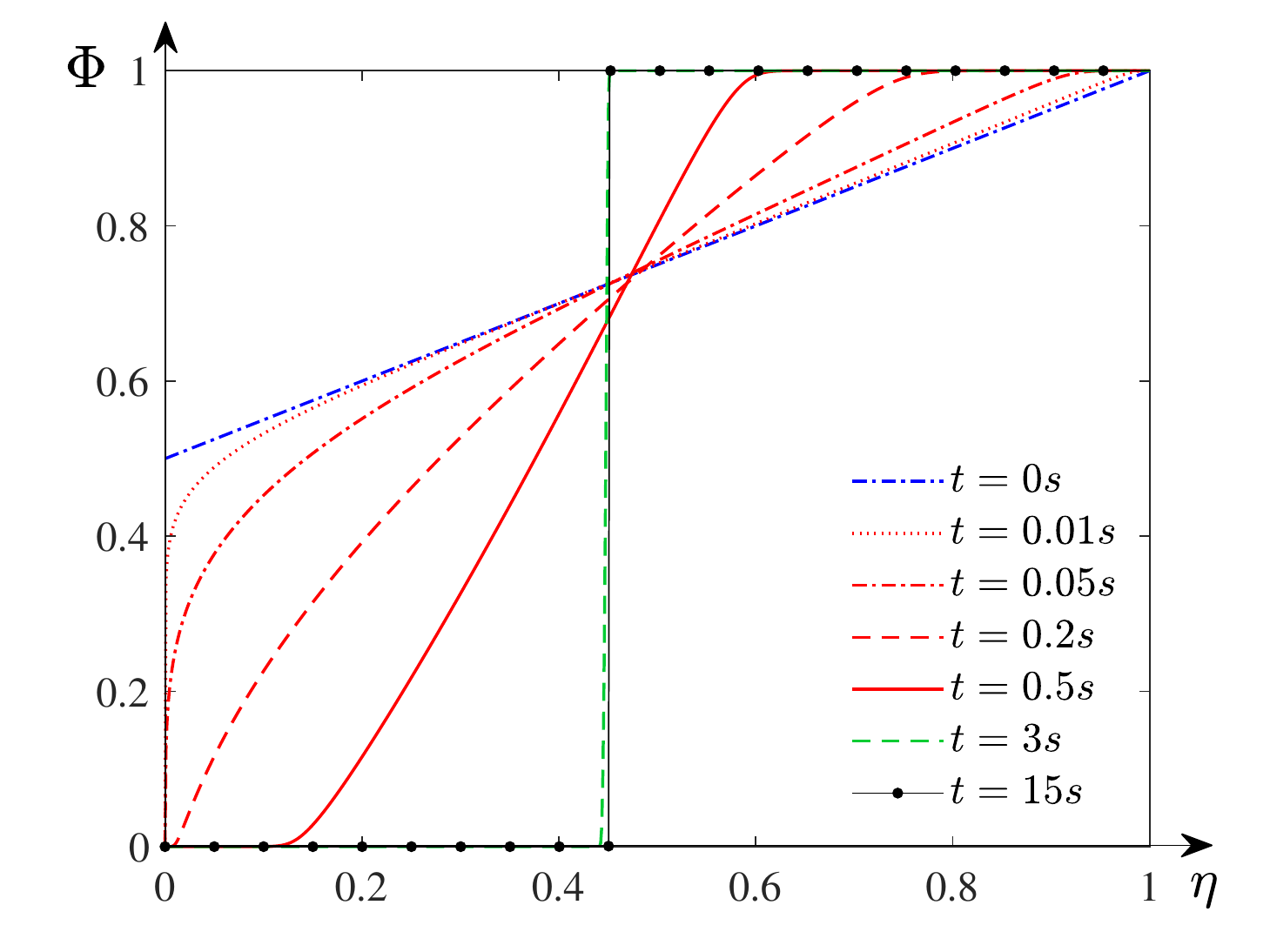}
  \end{subfigure}
  \begin{subfigure}{.327\textwidth}
    \centering
    \includegraphics[width =\linewidth]{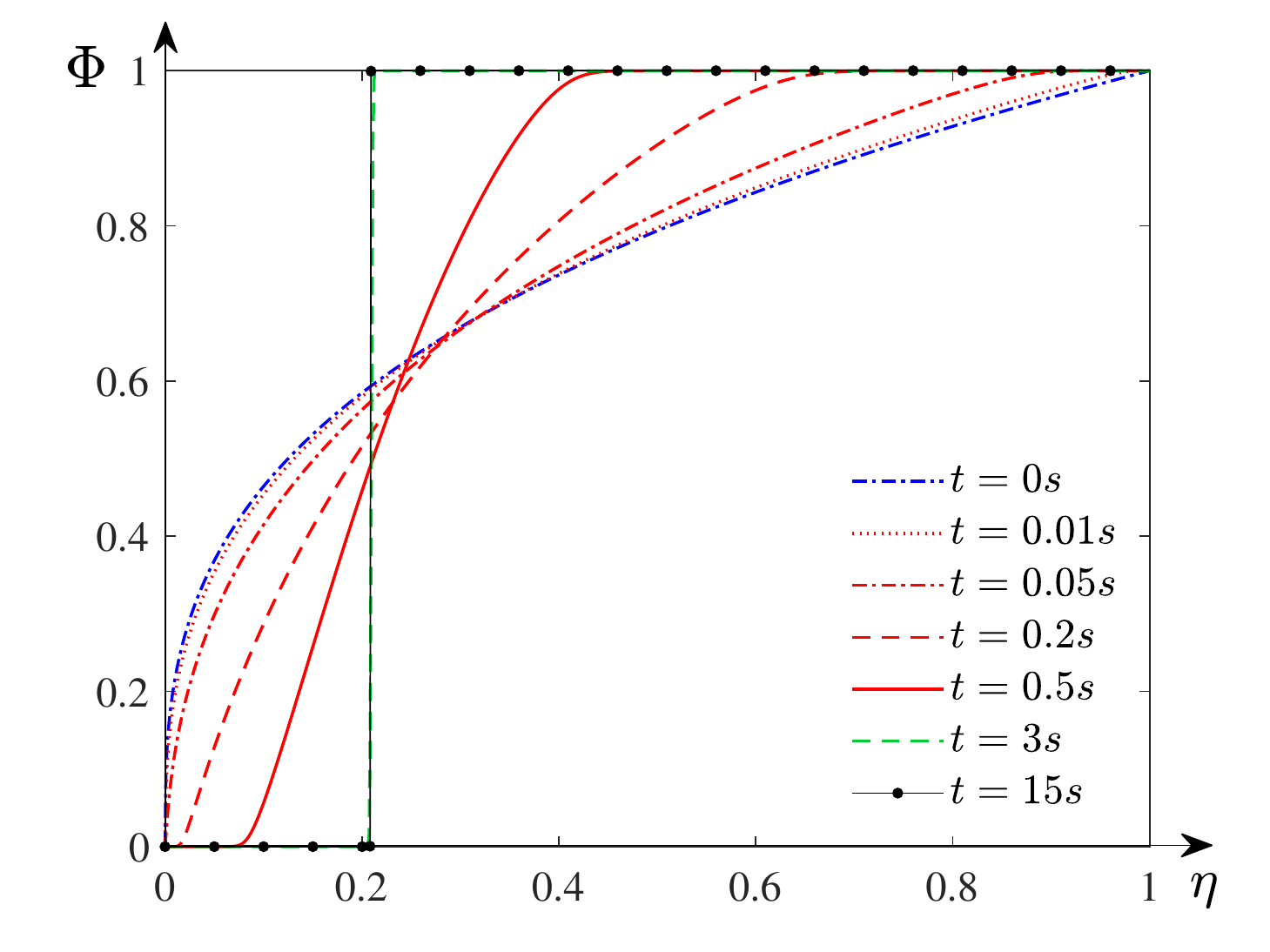}
    \caption{$f_0(x)=x^2$}
  \end{subfigure}%
  \hspace{0.01in}
  \begin{subfigure}{.327\textwidth}
    \centering
    \includegraphics[width = \linewidth]{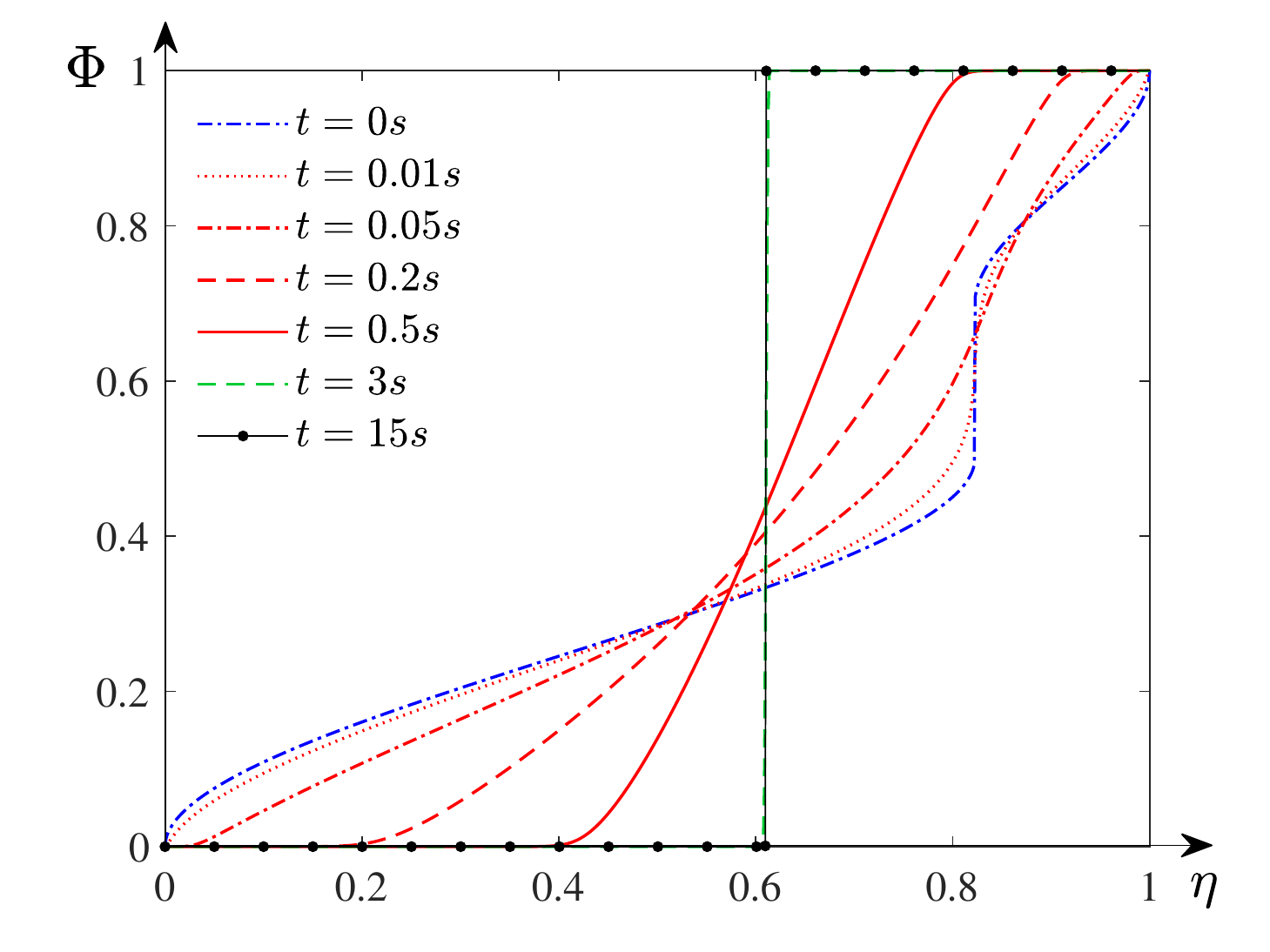}
    \caption{$f_0(x)=\max\big(0,\!x(0.5\!-\!x),\!(x\!-\!0.7)(1\!-\!x)\big)$}
  \end{subfigure}%
  \hspace{0.01in}
  \begin{subfigure}{.327\textwidth}
    \centering
    \includegraphics[width = \linewidth]{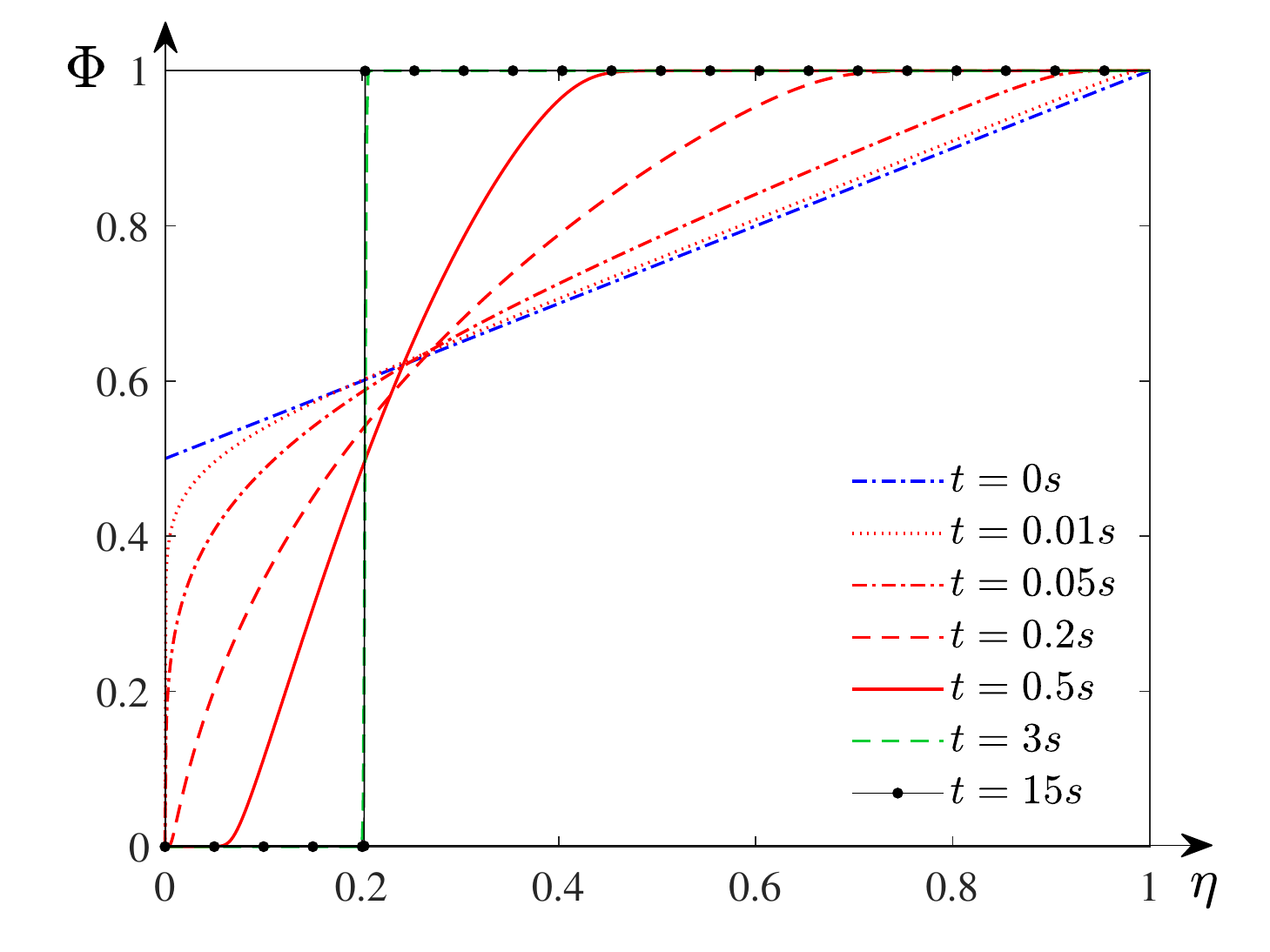}
    \caption{$f_0(x)=\mathbbm{1}_{[0.5,1]}$}
  \end{subfigure}
  \caption{Evolution of the full Replicator-Diffusion equation \eqref{driftdiff}.  \textbf{Top}: $\kappa=2$ and $V'(x)=2$. \textbf{Bottom}: $\kappa=2$ and $V'(x)=-3x+1$ for each of the specified initial data.}\label{fig:phi}
\end{figure}

Finally, we report in Table \ref{tabletheta0} the error of the computed jump location against the theoretical one given in Theorem \ref{theorem21}.  They add another layer in the robustness of our proposed scheme as the location in $\Phi$ determines the portion of the alleles in the genetic aggregation.
\renewcommand\arraystretch{1.1}
\begin{table}[H] 
\centering
\caption{Computed jump location for $\Phi_i^k$ up to a terminal time $T=15$.}
\begin{tabular}{lcccc}
\hline
initial data $f_0(x)$ & $V'(x)$  & Theoretical $\eta_0$ & Numerical $\tilde{\eta}_0$ & Error  \\ \hline
\multirow{2}{*}{$x^2$}                 & 2     & 0.4065      & 0.4324    & 0.0259 \\
       & -3$x$+1 & 0.2054      & 0.2092    & 0.0038 \\  \cline{2-5}
\multirow{2}{*}{$\max\big(0,\!x(0.5\!-\!x),\!(x\!-\!0.7)(1\!-\!x)\big)$}            & 2     & 0.7854      & 0.7948    & 0.0094 \\
          & -3$x$+1 & 0.6172      & 0.6106    & 0.0066 \\   \cline{2-5}
\multirow{2}{*}{$\mathbbm{1}_{[0.5,1]}$ }         & 2     & 0.4255      & 0.4515    & 0.0260 \\
         & -3$x$+1 & 0.1985      & 0.2032    & 0.0047 \\ \hline
\end{tabular}\label{tabletheta0}
\end{table}

\section*{Acknowledgments}
J.A. Carrillo was supported the Advanced Grant Nonlocal-CPD (Nonlocal PDEs for Complex Particle Dynamics: 	Phase Transitions, Patterns and Synchronization) of the European Research Council Executive Agency (ERC) under the European Union's Horizon 2020 research and innovation programme (grant agreement No. 883363). J.A. Carrillo was also partially supported by the EPSRC grant number EP/P031587/1. J.A. Carrillo acknowledges support through the Changjiang Visiting Professorship Scheme of the Chinese Ministry of Education.  Q. Wang is supported by Sichuan Science and Technology Program (No. 2020YJ0060). 

\bibliographystyle{
abbrv}
\bibliography{references}

\begin{thebibliography}{10}

\bibitem{gf-AGS}
L.~Ambrosio, N.~Gigli, and G.~Savar\'{e}.
\newblock {\em Gradient flows in metric spaces and in the space of probability
  measures}.
\newblock Lectures in Mathematics ETH Z\"{u}rich. Birkh\"{a}user Verlag, Basel,
  second edition, 2008.

\bibitem{BCC2008}
A.~Blanchet, V.~Calvez, and J.~A. Carrillo.
\newblock Convergence of the mass-transport steepest descent scheme for the
  subcritical patlak--keller--segel model.
\newblock {\em SIAM J. Numer. Anal.}, 46(2):691--721, 2008.

\bibitem{CKM2019}
J.~A. Carrillo, N.~Kolbe, and M.~Luk{\'a}{\v{c}}ov{\'a}-Medvid’ov{\'a}.
\newblock A hybrid mass transport finite element method for {K}eller--{S}egel
  type systems.
\newblock {\em J. Sci. Comput.}, 80(3):1777--1804, 2019.

\bibitem{CLSS10}
J.~A. Carrillo, S.~Lisini, G.~Savar\'{e}, and D.~Slep\v{c}ev.
\newblock Nonlinear mobility continuity equations and generalized displacement
  convexity.
\newblock {\em J. Funct. Anal.}, 258(4):1273--1309, 2010.

\bibitem{CMW20}
J.~A. Carrillo, D.~Matthes, and M.-T. Wolfram.
\newblock Lagrangian schemes for wasserstein gradient flows.
\newblock Handbook of Numerical Analysis. Elsevier, 2020.

\bibitem{CMV03}
J.~A. Carrillo, R.~J. McCann, and C.~Villani.
\newblock Kinetic equilibration rates for granular media and related equations:
  entropy dissipation and mass transportation estimates.
\newblock {\em Rev. Mat. Iberoamericana}, 19(3):971--1018, 2003.

\bibitem{CM2009}
J.~A. Carrillo and J.~S. Moll.
\newblock Numerical simulation of diffusive and aggregation phenomena in
  nonlinear continuity equations by evolving diffeomorphisms.
\newblock {\em SIAM J. Sci. Comput.}, 31(6):4305--4329, 2009/10.

\bibitem{CRW2016}
J.~A. Carrillo, H.~Ranetbauer, and M.-T. Wolfram.
\newblock Numerical simulation of nonlinear continuity equations by evolving
  diffeomorphisms.
\newblock {\em J. Comput. Phys.}, 327:186--202, 2016.

\bibitem{gf-CN}
F.~Cavalli and G.~Naldi.
\newblock A {W}asserstein approach to the numerical solution of the
  one-dimensional {C}ahn-{H}illiard equation.
\newblock {\em Kinet. Relat. Models}, 3(1):123--142, 2010.

\bibitem{Chalub}
F.~A. Chalub, L.~Monsaingeon, A.~M. Ribeiro, and M.~O. Souza.
\newblock Gradient flow formulations of discrete and continuous evolutionary
  models: a unifying perspective.
\newblock {\em arXiv preprint arXiv:1907.01681}, 2019.

\bibitem{Chalub2009}
F.~A. Chalub and M.~O. Souza.
\newblock From discrete to continuous evolution models: a unifying approach to
  drift-diffusion and replicator dynamics.
\newblock {\em Theor. Popul. Biol.}, 76(4):268--277, 2009.

\bibitem{Chalub2009a}
F.~A. Chalub, M.~O. Souza, et~al.
\newblock A non-standard evolution problem arising in population genetics.
\newblock {\em Commun. Math. Sci.}, 7(2):489--502, 2009.

\bibitem{DLWY2019}
C.~Duan, C.~Liu, C.~Wang, and X.~Yue.
\newblock Numerical complete solution for random genetic drift by energetic
  variational approach.
\newblock {\em ESAIM: M2AN}, 53(2):615--634, 2019.

\bibitem{Eyre}
D.~J. Eyre.
\newblock Unconditionally gradient stable time marching the cahn-hilliard
  equation.
\newblock In {\em Materials Research Society Symposium Proceedings}, volume
  529, pages 39--46. Materials Research Society, 1998.

\bibitem{Fisher1922}
R.~A. Fisher.
\newblock On the dominance ratio.
\newblock {\em Proceedings of the Royal Society of Edinburgh}, 42:321--341,
  1923.

\bibitem{HLP1934}
J.~E.~L. G.~H.~Hardy and G.~Polya.
\newblock {\em Inequalities}.
\newblock Cambridge University Press, Cambridge, 1934.

\bibitem{Gillespie2004}
J.~H. Gillespie.
\newblock {\em Population genetics: a concise guide}.
\newblock JHU Press, 2004.

\bibitem{gf-GT2}
L.~Gosse and G.~Toscani.
\newblock Identification of asymptotic decay to self-similarity for
  one-dimensional filtration equations.
\newblock {\em SIAM J. Numer. Anal.}, 43(6):2590--2606, 2006.

\bibitem{gf-GT1}
L.~Gosse and G.~Toscani.
\newblock Lagrangian numerical approximations to one-dimensional
  convolution-diffusion equations.
\newblock {\em SIAM J. Sci. Comput.}, 28(4):1203--1227, 2006.

\bibitem{gf-JKO}
R.~Jordan, D.~Kinderlehrer, and F.~Otto.
\newblock The variational formulation of the {F}okker-{P}lanck equation.
\newblock {\em SIAM J. Math. Anal.}, 29(1):1--17, 1998.

\bibitem{Kimura1955b}
M.~Kimura.
\newblock Random genetic drift in multi-allelic locus.
\newblock {\em Evolution}, pages 419--435, 1955.

\bibitem{Kimura1962}
M.~Kimura.
\newblock On the probability of fixation of mutant genes in a population.
\newblock {\em Genetics}, 47(6):713, 1962.

\bibitem{Kimura1964}
M.~Kimura.
\newblock Diffusion models in population genetics.
\newblock {\em J. Appl. Probab.}, 1(2):177--232, 1964.

\bibitem{Kimura1955a}
M.~Kimura et~al.
\newblock {\em Stochastic processes and distribution of gene frequencies under
  natural selection}.
\newblock Citeseer, 1954.

\bibitem{gf-MOfp}
D.~Matthes and H.~Osberger.
\newblock Convergence of a variational {L}agrangian scheme for a nonlinear
  drift diffusion equation.
\newblock {\em ESAIM Math. Model. Numer. Anal.}, 48(3):697--726, 2014.

\bibitem{gf-MOdlss}
D.~Matthes and H.~Osberger.
\newblock A convergent {L}agrangian discretization for a nonlinear fourth-order
  equation.
\newblock {\em Found. Comput. Math.}, 17(1):73--126, 2017.

\bibitem{Moran1962}
Moran and P.~A. Pierce.
\newblock {\em The statistical processes of evolutionary theory}.
\newblock Clarendon Press; Oxford University Press., 1962.

\bibitem{Moran}
P.~A.~P. Moran.
\newblock Random processes in genetics.
\newblock In {\em Mathematical proceedings of the cambridge philosophical
  society}, volume~54, pages 60--71. Cambridge University Press, 1958.

\bibitem{Nesterov}
Y.~Nesterov and A.~Nemirovskii.
\newblock {\em Interior-point polynomial algorithms in convex programming}.
\newblock SIAM, 1994.

\bibitem{gf-Osberger}
H.~Osberger.
\newblock Long-time behavior of a fully discrete {L}agrangian scheme for a
  family of fourth order equations.
\newblock {\em Discrete Contin. Dyn. Syst.}, 37(1):405--434, 2017.

\bibitem{gf-MOtf}
H.~Osberger and D.~Matthes.
\newblock Convergence of a fully discrete variational scheme for a thin-film
  equation.
\newblock In {\em Topological optimization and optimal transport}, volume~17 of
  {\em Radon Ser. Comput. Appl. Math.}, pages 356--399. De Gruyter, Berlin,
  2017.

\bibitem{Rice1987}
W.~R. Rice.
\newblock Genetic hitchhiking and the evolution of reduced genetic activity of
  the y sex chromosome.
\newblock {\em Genetics}, 116(1):161--167, 1987.

\bibitem{StarSpencer2013}
B.~Star and H.~G. Spencer.
\newblock Effects of genetic drift and gene flow on the selective maintenance
  of genetic variation.
\newblock {\em Genetics}, 194(1):235--244, 2013.

\bibitem{Tran2013}
T.~D. Tran, J.~Hofrichter, and J.~Jost.
\newblock An introduction to the mathematical structure of the wright--fisher
  model of population genetics.
\newblock {\em Theor. Biosci.}, 132(2):73--82, 2013.

\bibitem{Traulsen2005}
A.~Traulsen, J.~C. Claussen, and C.~Hauert.
\newblock Coevolutionary dynamics: from finite to infinite populations.
\newblock {\em Phys. Rev. Lett.}, 95(23):238701, 2005.

\bibitem{Vi03}
C.~Villani.
\newblock {\em Topics in optimal transportation}, volume~58 of {\em Graduate
  Studies in Mathematics}.
\newblock American Mathematical Society, Providence, RI, 2003.

\bibitem{Waxman2009}
D.~Waxman.
\newblock Fixation at a locus with multiple alleles: Structure and solution of
  the wright fisher model.
\newblock {\em J. Theor. Biol.}, 257(2):245--251, 2009.

\bibitem{gf-WW}
M.~Westdickenberg and J.~Wilkening.
\newblock Variational particle schemes for the porous medium equation and for
  the system of isentropic {E}uler equations.
\newblock {\em M2AN Math. Model. Numer. Anal.}, 44(1):133--166, 2010.

\bibitem{Wright1929}
S.~Wright.
\newblock The evolution of dominance.
\newblock {\em The American Naturalist}, 63(689):556--561, 1929.

\bibitem{Wright1937}
S.~Wright.
\newblock The distribution of gene frequencies in populations.
\newblock {\em Proc. Natl. Acad. Sci. U.S.A.}, 23(6):307, 1937.

\bibitem{Wright1945}
S.~Wright.
\newblock The differential equation of the distribution of gene frequencies.
\newblock {\em Proc. Natl. Acad. Sci. U.S.A.}, 31(12):382, 1945.

\bibitem{XCLZY20192019}
S.~Xu, M.~Chen, C.~Liu, R.~Zhang, and X.~Yue.
\newblock Behavior of different numerical schemes for random genetic drift.
\newblock {\em BIT Numer. Math.}, 59(3):797--821, 2019.

\bibitem{XCLY2019}
S.~Xu, X.~Chen, C.~Liu, and X.~Yue.
\newblock Numerical method for multi-alleles genetic drift problem.
\newblock {\em SIAM J. Numer. Anal.}, 57(4):1770--1788, 2019.

\end{thebibliography}
\end{document}